\documentclass{article}
\usepackage{amsmath}
\usepackage{xcolor}
\usepackage{graphicx}
\usepackage{latexsym}
\usepackage{amssymb}
\usepackage{tensor}
\usepackage{subcaption}
\newtheorem{thm}{Theorem}[section]
\newtheorem{la}[thm]{Lemma}
\newtheorem{Defn}[thm]{Definition}
\newtheorem{Remark}[thm]{Remark}
\newtheorem{prop}[thm]{Proposition}
\newtheorem{cor}[thm]{Corollary}
\newtheorem{Example}[thm]{Example}
\newtheorem{Number}[thm]{\!\!}

\newenvironment{rem}{\begin{Remark}\rm}{\end{Remark}}

\newenvironment{proof}{{\noindent\bf Proof.}}%
                  {\nopagebreak\hspace*{\fill}$\Box$\medskip\par}

\newcommand{\cL}{{\mathcal L}}
\newcommand{\cD}{{\mathcal D}}

\newcommand{\R}{{\mathbb R}}

\newcommand{\C}{{\mathbb C}}
\newcommand{\sgrad}{ \textrm{sgrad}}

\newcommand{\Sph}{{\mathbb S}}

\newcommand{\cg}{{\mathfrak g}}

\begin{document}
\begin{center}
{\bf\Large Curvature and stability of quasi-geostrophic motion  }\\[4mm]
{\bf Ali Suri}
\end{center}
\begin{abstract}
\hspace*{-4.7mm}
This paper outlines the study of the curvature of the quantomorphism group and its central extension, as well as the quasi-geostrophic equation. By utilizing spherical harmonics and structure constants, a formula for computing the curvature of the $L^2$ metric on the central extension $\hat{\cg}=\cg\ltimes_\Omega\mathbb{R}$ is derived, where $\cg$ represents the Lie algebra of $\cD^s_\mu(\mathbb{S}^2)$. The sectional curvatures of the planes containing $Y_{10}$ and the tradewind current are calculated as special cases. The impact of the Rossby and Froude numbers, as well as the Coriolis effect, on the (exponential) stability of these quasi-geostrophic motions is highlighted. Finally, a lower bound for weather prediction error in a simplified model governed by the tradewind current and the Coriolis effect on a rotating sphere is suggested.
\end{abstract}
%
%
%
%
%

%
%
%
%
\section{Introduction}
The study of conjugate points and the sign of sectional curvature has been an area of interest since Arnold's seminal work \cite{Arnold}. Arnold computed  the sectional curvature for the group of volume-preserving diffeomorphisms of the two-dimensional torus $\cD_\mu(\mathbb{T}^2)$ and showed that in many directions it was non-positive, leading to ideas about geodesic stability and weather forecast predictability duration \cite{Arnold-M}.

Milnor in \cite{Milnor} used structure constants and computed the sectional curvature of finite dimensional Lie groups with left invariant metrics.
Lukatskii \cite{Luk} continued the study of sectional curvature for $\cD^s_\mu(\mathbb{S}^2)$, obtaining similar estimates  for errors in weather forecast predictability. This was followed by various methods, including those proposed by Arakelyan and Savvidy \cite{Arak-Sav}, Dowker and Mo-Zheng \cite{Dowker-Zheng}, Yoshida \cite{Yoshida}, and the numerical study by Blender \cite{Blender}. Vizman \cite{Viz2001} considered the central extension of the group of volume-preserving diffeomorphisms of the torus to compute the sectional curvature of the configuration manifold for equations of charged fluids and superconductivity.

Misiolek  \cite{Misiolek-ams97, MisiolekJGP} considered the Bott-Virasoro group and studied the sectional curvature and conjugate points on the  one dimensional central extension of group of volumorphisms of circle.

Ebin and Preston \cite{Ebin-Preston, Ebin-Preston-arXiv} proved that the Euler-Arnold equations on the one-dimensional central extension of $T_e\cD(\mathbb{S}^3)$ reduce to the quasi-geostrophic equations on $\mathbb{S}^2$. Recently, Lee and Preston \cite{Lee-Pre} used this framework to study the sectional curvature, conjugate points, and stability of geodesics  on the one-dimensional central extension of the quantomorphism group. The effects of the Coriolis force on conjugate points and stability were pointed out by Tauchi and Yoneda \cite{Tauchi-Yoneda-2} also.

We propose a formula that computes the sectional curvature using spherical harmonics and the Lie algebra structure constants. Motivated by the works of Arakelyan and Savvidy \cite{Arak-Sav}, Dowker, Mo-Zheng \cite{Dowker-Zheng}, and Yoshida \cite{Yoshida}, we compute the curvature of the $L^2$ metric on $\hat\cg=\cg\ltimes_\Omega\mathbb{R}$, where $\cg=T_e\cD_\mu(\mathbb{S}^2)$ and $\Omega$ is the two-cocycle used in the central extension. Our results reproduce those of \cite{Arnold}, \cite{Luk}, \cite{Dowker-Zheng} and  \cite{Yoshida} if we restrict to the Lie algebra $\cg$. As special cases, we study the sectional curvatures of planes containing $Y_{1~0}$, the tradewind current, and other vector fields generated by spherical harmonics. Building on the works of Arnold \cite{Arnold}, Lukatskii \cite{Luk}, Dowker and Mo-Zheng \cite{Dowker-Zheng}, and Yoshida \cite{Yoshida}, we extend their findings by taking into account the effects of rotation, and compare our results with Arnold and Lukatskii's estimates of the unreliability of long-term weather forecasts. The stability effects of the Froude and Rossby numbers (rotation), as confirmed by Lee and Preston, are also taken into consideration.

Finally we provide an estimation of the error growth in the weather prediction for a homogeneous incompressible fluid modeled on the sphere $\mathbb{S}^2$, approximated by the tradewind current.
We will see that  rotation can  make the predictability period longer.
%
%
%
%

\subsection{Motivations}
The curvature properties of $\mathcal{D}^s_{\text{vol}}(\Sph^s)$ have been the subject of study in several works, such as those by Arakelyan and Savvidy, Benn, Blender, Dowker-Mo zheng, Lee-Preston, Misiolek, Tauchi-Yoneda, and Yoshida \cite{Arak-Sav, Benn2021, Blender, Dowker-Zheng, Lee-Pre, Misiolek, Tauchi-Yoneda, Yoshida}.

On the other hand  since the Coriolis force plays a crucial role in quasi-geostrophic  motion, it becomes meaningful to investigate the curvature of the volumorphism group of a rotating sphere \cite{Lee-Pre,Tauchi-Yoneda-2}.

Motivated by the works   \cite{Arak-Sav,Yoshida}, our first goal is to present a computable curvature formula for the current situation using the structure constants introduced by Hoppe \cite{Hoppe}. Despite the fact that the Wigner $3j$-symbols used in the definition of structure constants for spherical harmonics  (see A13 in \cite{Arak-Sav}) are less known in geometry books, there exists a remarkable literature about them in quantum mechanics, including works by Messiah \cite{Mess} and Varshalovich etc. \cite{Varsha}.

However, there is a growing interest in this approach due to the works of Modin, Viviani and Cifani, who employed  finite dimensional approximations of the Lie algebra $T_e\mathcal{D}^s_{\text{vol}}(\Sph^2)$, using finite dimensional algebras produced by matrix harmonics and discrete Laplacian (see e.g.  \cite{Modin, Modin2023} and the references therein).

Computing curvature and determining the sign of the sectional curvature are usually difficult tasks.
While our curvature formula can be used directly, as shown in section \ref{section sect. comp},  one can use  theorem \ref{thm curv e} to create a software-computable program  similar to the approach used by Blender in \cite{Benn2021}.

However, the computations presented in section \ref{section sect. comp} are necessary to establish the compatibility of our approach with previous works, such as those by Arakelyan and Savvidy, Dowker-Mo zheng, and especially Lukatskii \cite{Arak-Sav, Dowker-Zheng, Luk}. On the other hand, the tradewind current is intrinsically generated by the Coriolis effect. Arnold \cite{Arnold-M}, Lukatskii \cite{Luk}, and Yoshida \cite{Yoshida} presented interesting results regarding the unreliability of long-term weather forecasts when the Coriolis effect is not taken into account. As an application, we utilize our method to impose the Coriolis force on existing results.

%
%
%
%

\subsection{Fluid dynamics backgrounds}
Earth is in constant rotation around its axis from west to east, with points closer to the equator spinning faster than those near the poles. When a fluid particle moves along a meridian to  north in the northern hemisphere, its trajectory is deflected to the right due to the conservation of spinning momentum.
Conversely, in the southern hemisphere, particles moving away from the equator along a meridian are deflected to the left.
The Coriolis force (effect), is a force which in the northern hemisphere deflects  moving air (fluid)  to the right of its intended path and in the southern hemisphere  deflects it to the left. The Coriolis  force magnitude  changes depending  on the distance from equator.

There are large-scale atmospheric currents known as tradewinds, which result from temperature differences between the equator and the poles and the Coriolis effect.
A simple model for the tradewind current was studied by Arnold in \cite{Arnold} on a torus and by Lukatskii in \cite{Luk} on the sphere. This model is described by the velocity field $\mu(-y\partial_x +x\partial_y)$ (see figure 1), which is induced by the Coriolis force and uneven heating of the sphere.
\begin{figure}[!h]
\centering
  \includegraphics[width=7cm]{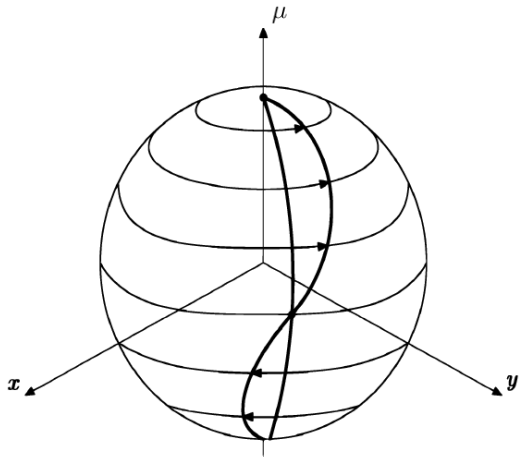}
  \caption{The tradewind current}\label{fig:1}
  \label{fig:boat1}
\end{figure}

Quasi-geostrophic motion refers to a state where the Coriolis force is of the same order of magnitude as the horizontal pressure gradient. It is important to note that while the vertical pressure gradient is a result of gravity, the horizontal pressure gradient primarily arises from the uneven heating of the earth's surface. The quasi-geostrophic equation in the $\beta$-plane approximation is given by
\begin{equation}\label{eq geodesic-equation beta plane}
\partial_t(\Delta f-\alpha^2 f) +  \{f,\Delta  f\} +\beta f_\lambda  =0.
\end{equation}
(compare with \eqref{eq geodesic-equation of central extension reduced}) assumes that the Coriolis parameter varies linearly with latitude as $f=f_0+\beta\mu$ where $f_0$ and $\beta$ are constants. The parameter (number) $\beta$ is known as the Rossby number. In a more restrictive situation known as the "$f$-plane approximation," we replace the variable $f$ with a constant.

The parameter $\alpha$ in \eqref{eq geodesic-equation beta plane} is a dimensionless parameter called the Froude number. Froude number is defined by speed-length ratio and it is given by
\begin{equation}\label{Froude number}
\alpha:=\frac{U}{\sqrt{gL}}
\end{equation}
where $U$ is the horizontal velocity scale, $L$ is the horizontal length scale, and $g$ is the gravitational constant. In the large scale, when $U \ll L$, the Froude number becomes very small, i.e., $\alpha \ll 1$.

As suggested by Ebin and Preston in \cite{Ebin-Preston-arXiv}, the implementation of the parameter $\alpha$ involves considering the Hopf fibration, with the Coriolis effect being imposed as a central extension. However, an alternative approach is to consider a two-dimensional central extension of $\mathcal{D}^s_{vol}(\Sph^2)$ without the need to deal with the Hopf fibration and the quantomorphisms group \cite{Lee-Pre}.

%
%
%
%
\section{Hopf fibration and quasi-geostrophic equations on $\Sph^2$}
Let $SU(2)$ be the set of all $2\times 2$ complex matrices of the form
$A=\begin{pmatrix}
       z  &       w  \\
-\bar{w}  &  \bar{z}
\end{pmatrix}$
where $z,w\in \mathbb{C}$ and $\textrm{det}(A)=1$. The corresponding Lie algebra $\mathfrak{su}(2)$ is the set of all matrices with the form
$\begin{pmatrix}
 i\alpha  &       \beta+i\gamma  \\
-\beta+i\gamma  &  -i\alpha
\end{pmatrix}$
where $\alpha,\beta,\gamma$ are real numbers. Consider the basis of orthogonal vectors
\begin{equation*}
E_1=\begin{pmatrix}
i  &  0  \\
0  &  -i
\end{pmatrix}, \quad
E_2=\begin{pmatrix}
0  &  1  \\
-1  &  0
\end{pmatrix},\quad
E_3=\begin{pmatrix}
0  &  i  \\
i  &  0
\end{pmatrix}
\end{equation*}
for $\mathfrak{su}(2)$ with the (Berger) metric coefficients $g_{11}=\alpha^2$, $g_{22}=1$ and $g_{33}=1$. We identify $SU(2)$ with the 3-dimensional unit sphere $\Sph^3(1)$. The Hopf fibration  is a Riemannian submersion defined  by
\begin{eqnarray*}
H:\Sph^3(1)   &\longrightarrow&    \Sph^2(\frac{1}{2})\subseteq \mathbb{R}\oplus\mathbb{C}\\
(z,w)      &\longmapsto&  \big(\frac{1}{2}(|z|^2-|w|^2),z\bar{w}\big)
\end{eqnarray*}
and the action $A$ of $\Sph^1$ on $SU(2)$ is given by $e^{i\theta}(z,w)=(e^{i\theta}z  ,e^{i\theta}w)$.
In this case $H$ delivers a fibration with the typical fiber $\Sph^1$. For the left invariant vector fields generated by $\{E_i\}_{1\leq i\leq3}$ we consider the Lie bracket defined by $[E_i,E_{i+1}]=2E_{i+2}$ which the indices are considered mod $3$. We consider $E_1$ as the Reeb vector field and the associated  1-form is denoted by $\theta=\omega_1$. For the associated dual basis $\{\omega_i\}_{1\leq i\leq 3}$ we have $d\theta=2\omega_2\wedge\omega_3$ (see e.g. \cite{Petersen}). For $X,Y\in T_x\Sph^2$ the $2$-form $\omega_x(X,Y)=\textrm{det}(x,X,Y)$ is the standard symplectic structure on $\Sph^2$ and  the pullback of $\omega$ is (a constant multiple of) $d\theta$. In fact Hopf fibration is an example of Boothby-Wang fibration. Ebin and Preston in \cite{Ebin-Preston, Ebin-Preston-arXiv} using the Riemannian geometry of the quantomorphism group of $\Sph^3$ and the Boothby-Wang fibration proposed an approach to study the quasi-geostrophic equations on $\Sph^2$ as we review here.

The qunatomorphism group $\cD^s_q(\Sph^3)=\{\eta\in\cD^s(\Sph^3);~\eta^*\theta=\theta\}$ for $s>\frac{3}{2}+1$, admits a smooth manifold structure (Corollary 2.7 \cite{Ebin-Preston-arXiv}). The tangent space is
\begin{equation*}
\cg:=T_e \cD^s_q(\Sph^3)=\{ S_\theta f;~f\in \mathcal{F}_{E_1}^{s+1}(\Sph^3,\mathbb{R})\}
\end{equation*}
where $\mathcal{F}_{E_1}^{s+1}(\Sph^3,\mathbb{R})=\{f:\Sph^3\longrightarrow \mathbb{R};~ \textrm{f is~} H^{s+1} \textrm{~and~} E_1(f)=0\}$ and the operator $S_\theta$ is defined by the following properties
\begin{equation*}
u=S_\theta f \quad \Longleftrightarrow  \quad \theta(u)=f~~ \textrm{and}~~i_ud\theta=-df.
\end{equation*}
In this case $S_\theta f=fE_1   -  \frac{1}{2}(E_3f)E_2   +  \frac{1}{2}(E_2f)E_3$ and
\begin{equation*}
\Delta_\theta f=\alpha^2  -  \frac{1}{4}E_2^2  -  \frac{1}{4}E_3^2
\end{equation*}
where $\Delta_\theta=S_\theta^* S_\theta$ is the contact Laplacian and $S_\theta^*$ is the adjoint of the $S_\theta$ with respect to the right invariant $\cL^2$-metric
\begin{equation*}
\ll S_\theta f, S_\theta g \gg_\cg=\int_{\Sph^3}<S_\theta f,S_\theta g>d\mu.
\end{equation*}
(see also \cite{Ebin-Preston-arXiv}, pp. 20-23). For any $f,g\in\mathcal{F}_{E_1}^{s+1}(\Sph^3,\mathbb{R})$ the contact Poisson bracket is defined by the relation $\{f,g\}=(S_\theta f) g$. In this case we have $S_\theta\{f,g\}=[S_\theta f,S_\theta g]$ which means that $S_\theta$ is a Lie algebra morphism.
%
%

The following lemma is true for any contact manifold $M$ and as an special case for $M=\Sph^3$. The Riemannian metric on $M$ is denoted by $<,>$.
\begin{la}
The coadjoint operator $ad^*_u:\cg\rightarrow\cg$ is given by
\begin{equation}\label{eq ad* on quantomorphism}
ad^*_uv=S_\theta\Delta_\theta^{-1} \{f,\Delta_\theta g\}
\end{equation}
where $u=S_\theta f$, $v=S_\theta g$.
\end{la}
For the proof see \cite{Ebin-Preston-arXiv}.\\
%
%
As a result, the Euler-Arnold equation on $\cD^s_q(M)$ is given by
\begin{eqnarray*}
0=\partial_tu+ad^*_uu   &=&    \partial_t S_\theta f + ad^*_{S_\theta f}S_\theta f\\
&=&  \partial_t S_\theta f + S_\theta\Delta_\theta^{-1}\{f,\Delta_\theta f\}\\
&=& S_\theta \Big(  \partial_t   f +  \Delta_\theta^{-1} \{f,\Delta_\theta f\}  \Big)
\end{eqnarray*}
which implies that  $\partial_t   f +  \Delta_\theta^{-1} \{f,\Delta_\theta f\}=0$. Now we apply the contact Laplacian on both sides of the last equation and we get
\begin{equation}\label{eq geodesic eq on D_q with L^2 metric}
\partial_t\Delta_\theta f +  \{f,\Delta_\theta f\}=0.
\end{equation}
See also \cite{Ebin-Preston-arXiv} theorem 4.1 for a different approach to this equation.
%
%
%
\subsection{Central extension of the quantomorphism group}
Following \cite{Ebin-Preston} and \cite{Ebin-Preston-arXiv} we consider the central extension of the Lie Algebra $\cg=T_e\cD^s_q(M)$ with $\mathbb{R}$ which is denoted  by $\hat\cg=\cg\ltimes_\Omega\mathbb{R}$. For $u=S_\theta f$ and $v=S_\theta g$ in $\cg$ the map $ \Omega(u,v)=\int_{\Sph^3}\phi\{f,g\}d\mu$ and $\phi:M\longrightarrow\mathbb{R}$ is a known function. Recall that for any $(u,a),(v,b)\in \cg\ltimes_\Omega\mathbb{R}$  the Lie bracket is defined by
\begin{equation*}
\big[  (u,a),(v,b)\big]=\big(  S_\theta\{f,g\}  ,  \Omega(u,v)  \big)
\end{equation*}
and the inner product is
\begin{equation*}
\ll  (u,a),(v,b)  \gg_{\hat{\cg}}=\int_{\Sph^3} <S_\theta f  ,  S_\theta g>d\mu  + ab.
\end{equation*}
First we calculate the operator $T:\cg\longrightarrow \cg$ defined by the relation $\ll Tu,v \gg=\Omega(u,v)$. Using the fact that $\int_{\Sph^3}\phi\{f,g\}d\mu = \int_{\Sph^3} \{\phi,f\}  g  d\mu$ we have
\begin{eqnarray*}
\Omega(u,v)   &=&   \int_{\Sph^3}\phi\{f,g\}d\mu\\
&=&    \int_{\Sph^3} \{\phi,f\}  g  d\mu\\
&=&    \int_{\Sph^3} S_\theta^*S_\theta\Delta_\theta^{-1}\{\phi,f\}  g  d\mu\\
&=&    \int_{\Sph^3} <  S_\theta\Delta_\theta^{-1}\{\phi,f\}  ,   S_\theta g > d\mu\\
&=&   \ll  S_\theta\Delta_\theta^{-1}\{\phi,f\}  ,   S_\theta g \gg_\cg
\end{eqnarray*}
which implies that $T(S_\theta f)=S_\theta\Delta_\theta^{-1}\{\phi,f\}$.\\[2mm]
%
%
%
Following \cite{Viz1}, for $\hat{ad}_{(u,a)}^*:\hat{\cg}\longrightarrow \hat{\cg}$ we have
\begin{eqnarray*}
\ll  \hat{ad}_{(u,a)}^* (v,b)  ,  (w,c) \gg_{\hat{\cg}}    &=&    \ll   (v,b)  , \hat{ad}_{(u,a)} (w,c) \gg_{\hat{\cg}}   \\
&=&  -  \ll   (v,b)  , [(u,a), (w,c)]_{\hat{\cg}} \gg_{\hat{\cg}}\\
&=&    \ll   (v,b)  , (ad_uw, -\Omega(u,w)) \gg_{\hat{\cg}}\\
&=&       \ll v, ad_uw\gg_\cg   - b\Omega(u,w) \\
&=&      \ll ad^*_uv  ,   w  \gg_\cg   - \ll b T u  ,  w  \gg_\cg   \\
&=&  \ll  ad^*_uv  - b Tu   ,   w  \gg_\cg.
\end{eqnarray*}
As a consequence  $  \hat{ad}_{(u,a)}^* (v,b)=(ad^*_uv  - b Tu,0)\in\hat{\cg}$  and  for the curve $(u,a):(-\epsilon,\epsilon)\longrightarrow \hat{\cg}$  the Euler-Arnold  equation is given by
\begin{equation*}
\left\{ \begin{array}{ll}  \partial_tu  +  ad^*_uu  -  a(t)Tu  =  0   \\
\partial_ta(t)=0
\end{array}\right.
\end{equation*}
The second equation implies that $a(t)=a$ is constant and following the procedure for derivation of equation (\ref{eq geodesic eq on D_q with L^2 metric}) for
$u=S_\theta f$ we have
\begin{equation}\label{eq geo-eq of cent-ext pre}
\partial_t\Delta_\theta f  +  \{f,\Delta_\theta f\}  -  a\{\phi,f\}  =  \partial_t\Delta_\theta f  +  \{f,\Delta_\theta f  + a\phi\} =  0
\end{equation}
%
%
\begin{rem}
Since for any $f\in\mathcal{F}_{E_1}^{s+1}(\Sph^3,\mathbb{R})$ we have $E_1f=0$ basically the function $f$ is constant in the direction of the Reeb field  (let's say with respect to the last variable in local chart of $\Sph^3$) then integration on $\Sph^3$ reduces to integration on the symplectic quotient $\Sph^2$. In this case  $\Delta_\theta=\alpha^2-\Delta$ where $\Delta$ is the usual Laplacian $\Sph^2$ and $S_\theta$ will reduced to the skew-gradient. Moreover the quasi-geostrophic  equations (\ref{eq geodesic eq on D_q with L^2 metric}) and (\ref{eq geo-eq of cent-ext pre}) are given by
\begin{equation}\label{eq geodesic eq on D_q with L^2 metric reduced}
\partial_t(\Delta f-\alpha^2 f) +  \{f,\Delta  f\}=0
\end{equation}
and
\begin{equation}\label{eq geodesic-equation of central extension reduced general}
\partial_t(\Delta f-\alpha^2 f) +  \{f,\Delta  f\}  - a\{f, \phi\}=0.
\end{equation}
respectively (compare with \cite{Lee-Pre} Corollary 5). In meteorology the function $\phi$ is locally given by $\phi(\lambda,\mu)=\mu$  and $\{f,g\}$ is the Poisson bracket which in Darboux coordinates resembles
\begin{equation}\label{poisson bracket preston}
\{f,g\}=   \frac{\partial f}{\partial \lambda}\frac{\partial g}{\partial \mu}   -   \frac{\partial f}{\partial \mu} \frac{\partial g}{\partial \lambda} \end{equation}
As a result we deal with the following equation (see also \cite{Skiba} for the case that $\alpha^2=0$.)
\begin{equation}\label{eq geodesic-equation of central extension reduced}
\partial_t(\Delta f-\alpha^2 f) +  \{f,\Delta  f-a\mu\}  =0.
\end{equation}
\end{rem}
We will use the symbol "$\sgrad$" for the reduced case of $S_\theta$ on  $\mathbb{S}^2$ which  stands for the skew-gradient.
%
%
%
%
%
%
%
%
%
%
%
\section{Curvature of the central extension of the quantomorphism group}\label{section sec curv}
In this section first, we briefly introduce spherical harmonics and the corresponding structure constants from \cite{Mess, Arak-Sav}. Then we compute the cocycle $\Omega$ and the operator $T$, coadjoint operator and the covariant derivative according to the structure constants. Theorem \ref{thm curv e} is the main goal of this part and presents a formula which enables us to compute the curvature at the presence of Coriolis force and Froude number. It is tried to present the results in a way that after restricting to the non-rotating case, the results of \cite{Arak-Sav} appear.\\
Consider the complex spherical harmonic $Y_{lm}:\Sph^2\rightarrow \C$ where
\begin{equation*}
Y_{lm}(\lambda,\mu)=C^m_l P^{|m|}_l(\mu)e^{im\lambda}
\end{equation*}
and for the integers  $0\leq |m|\leq l$ the coefficient  $C^m_l=(-1)^{m}\sqrt{   \frac{2l+1}{4\pi}\frac{(l-|m|)!}{(l+|m| )!}  }$. Following the notations of \cite{Mess} the associated Legendre polynomial is given by
\begin{equation*}
P^{|m|}_l(\mu)=\frac{  {(1-\mu^2)}^{\frac{m}{2}}   }{2^ll!}  \frac{d^{l+|m|}}{d\mu^{l+|m|}}  (\mu^2-1)^l
\end{equation*}
with $-1\leq \mu\leq 1$ and $0\leq\lambda\leq 2\pi$. Moreover the complex conjugation is given by
$$Y_{lm}^\star=(-1)^mY_{l-m}.$$

Now suppose that $e_{l_km_k}=\sgrad(Y_{l_km_k})$ for $k\in\{1,2\}$ and $f,g:\Sph^2\longrightarrow \R$ be differentiable. Then due to the facts
\begin{equation}\label{inner pproduct}
\langle   Y_{l_1m_1}  ,   Y_{l_2m_2}    \rangle  = \int_{-1}^{1}\int_0^{2\pi} Y_{l_1m_1}  Y_{l_2m_2} d\lambda d\mu  =  (-1)^{m_1} \delta^{l_1}_{l_2}\delta^{m_1}_{-m_2},
\end{equation}
\begin{equation*}
\ll \sgrad f,\sgrad g\gg_\cg=\langle  (\alpha^2 -\Delta)f ,  g \rangle
\end{equation*}
and $\Delta Y_{lm}=-l(l+1)Y_{lm}$ we have
\begin{eqnarray*}
\ll  (e_{l_1m_1},a)  ,   (e_{l_2m_2},b)    \gg_{\hat\cg}    & = &  \ll  e_{l_1m_1}  ,   e_{l_2m_2}    \gg_{\cg}   +  ab\\
& = &  \ll  \sgrad  (Y_{l_1m_1})  ,   \sgrad(Y_{l_2m_2})    \gg_{\cg}   +  ab\\
& = &  \langle  (\alpha^2 -\Delta) (Y_{l_1m_1})  ,   Y_{l_2m_2}    \rangle   +  ab\\
& = &  \langle   (\alpha^2 +  l_1(l_1+1))  Y_{l_1m_1}  ,   Y_{l_2m_2}    \rangle   +  ab\\
& = &  \big(  \alpha^2 +  l_1(l_1+1) \big)   (-1)^{m_1} \delta^{l_1}_{l_2}\delta^{m_1}_{-m_2}   +  ab.\\
\end{eqnarray*}
In the case that $\alpha=a=b=0$ then the above equality reduces to equation (11) in \cite{Arak-Sav}. Let
\begin{equation}\label{structure constants}
\{Y_{l_1m_1},Y_{l_2m_2}\}:=G^{l_3m_3}_{l_1m_1l_2m_2} Y_{l_3m_3}
\end{equation}
or equivalently $[e_{l_1m_1},e_{l_2m_2}]:=G^{l_3m_3}_{l_1m_1l_2m_2} e_{l_3m_3} $ where we used the Einstein summation convention.
%
%
\begin{la}\label{lemma  Te}
Let $v=e_{l_2m_2}$ then
\begin{equation}\label{eq Te^m_l}
Tv=\sqrt{\frac{4\pi}{3}} \frac{G^{l_3m_3}_{1~0 ~l_2 m_2}}{\alpha^2+l_3(l_3+1)}    e_{l_3m_3}   :=     D^{l_3m_3}_{1~0 ~l_2 m_2}e_{l_3m_3}.
\end{equation}
\end{la}
%
%
\begin{proof}
Note that
$\phi=\mu=\sqrt{\frac{4\pi}{3}}Y_{10}$.
Following the formalism of \cite{Arak-Sav}
\begin{equation*}
\ll Tv,e_{l_3m_3} \gg_\cg =(-1)^{m_3} D^{l_3 -m_3}_{1~0~l_2m_2}  \Longleftrightarrow Tv= D^{l_3 m_3}_{1~0~l_2m_2}e_{l_3 m_3}.
\end{equation*}
Then we have
\begin{eqnarray*}
\ll Tv,e_{l_3 m_3} \gg_\cg   &=&   \ll   \sgrad\Delta_\theta^{-1} \{  \phi, Y_{l_2m_2}  \}  ,  \sgrad(Y_{l_3 m_3})   \gg_\cg \\
&=&   \sqrt{  \frac{4\pi}{3}  }   \langle    \{Y_{10}, Y_{l_2 m_2} \}  ,  Y_{l_3 m_3}   \rangle   \\
&=&   \sqrt{  \frac{4\pi}{3}  }  G^{l_1m_1}_{1~0~l_2m_2} \langle      Y_{l_1 m_1}  ,  Y_{l_3m_3}   \rangle   \\
&=&   \sqrt{  \frac{4\pi}{3}  }     (-1)^{m_3}   G^{l_1m_1}_{1~0~l_2m_2}    \delta^{m_1}_{-m_3}\delta^{l_1}_{l_3}\\
&=&   \sqrt{  \frac{4\pi}{3}  }  (-1)^{m_3}      G^{l_3-m_3}_{1~0~l_2m_2} \\
\end{eqnarray*}
On the other hand we have
\begin{eqnarray*}
\ll Tv,e_{l_3 m_3} \gg_\cg   &=&   \ll   D^{lm}_{1~0 ~l_2 m_2}e_{lm}  ,    e_{l_3 m_3}   \gg_\cg \\
&=&     D^{lm}_{1~0 ~l_2 m_2}   \ll   e_{lm}  ,    e_{l_3 m_3}   \gg_\cg \\
&=&       D^{lm}_{1~0 ~l_2 m_2}   (\alpha^2+l_3(l_3+1))   \langle  Y_{lm}  ,    Y_{l_3 m_3}   \rangle \\
&=&       D^{lm}_{1~0 ~l_2 m_2}  (-1)^{m_3} (\alpha^2+l_3(l_3+1))   \delta_l^{l_3}\delta_m^{-m_3} \\
&=&       D^{l_3-m_3}_{1~0 ~l_2 m_2}  (-1)^{m_3} (\alpha^2+l_3(l_3+1))  \\
\end{eqnarray*}
The last two equations imply that  $D^{l_3m_3}_{1~0 ~l_2 m_2}=  \sqrt{  \frac{4\pi}{3}  }  \frac{G^{l_3m_3}_{1~0~l_2m_2}  }{\alpha^2+l_3(l_3+1)}$.
\end{proof}
%
%
%

We will use the real structure constants
\begin{equation}\label{real structure constants}
G^{l_3m_3}_{l_1m_1l_2m_2}=-i(-1)^{m_3}g^{l_3-m_3}_{l_1m_1l_2m_2}
\end{equation}
and equalities
\begin{equation}\label{eq symmetry}
g^{l_3m_3}_{l_1m_1l_2m_2}=g^{l_2m_2}_{l_3m_3l_1m_1 }=g^{l_1m_1}_{l_2m_2l_3m_3}
\end{equation}
and
\begin{equation}\label{eq g 10}
g^{10}_{l_1m_1l_2m_2}=(-1)^{m_1}m_1\sqrt{\frac{3}{4\pi}} \delta^{l_1}_{l_2} \delta^{m_1}_{-m_2}
\end{equation}
from \cite{Arak-Sav}.
%
%
\begin{rem}
When it is needed, one can choose the following alternative for $Tv$. Using equations \eqref{real structure constants}, \eqref{eq symmetry} and \eqref{eq g 10} we have
\begin{equation*}
G^{l_3m_3}_{1~0~l_2m_2}   = -i(-1)^{m_3}g^{l_3-m_3}_{1~0 ~ l_2m_2} =  -i(-1)^{m_3}g^{1~0}_{ l_2m_2 l_3-m_3}     =-im_2\sqrt{\frac{3}{4\pi}} \delta^{l_3}_{l_2} \delta^{m_3}_{m_2}
\end{equation*}
which implies that
\begin{equation*}
Tv=\frac{-im_2}{\alpha^2+l_2(l_2+1)}e_{l_2m_2}.
\end{equation*}

\end{rem}
%
%
%
%
\begin{cor}\label{Cor Omega(e,e)}%
With the above notations we have
\begin{eqnarray*}
\Omega(e_{l_2m_2},e_{l_3m_3})  =  \ll Tv,e_{l_3 m_3} \gg_\cg   &=&   \sqrt{  \frac{4\pi}{3}  }  (-1)^{m_3}      G^{l_3-m_3}_{1~0~l_2m_2} \\
&=&   \sqrt{  \frac{4\pi}{3}  }  (-i)      g^{l_3m_3}_{1~0~l_2m_2}\\
&=&   \sqrt{  \frac{4\pi}{3}  }  (-i)      g^{10}_{l_2m_2l_3m_3}\\
&=&   \sqrt{  \frac{4\pi}{3}  }  (-i)(-1)^{m_2}m_2\delta^l_{l_3}\delta^{m_2}_{-m_3} \sqrt{  \frac{3}{4\pi}  }\\
&=&   -i(-1)^{m_2}m_2\delta^{l_2}_{l_3}\delta^{m_2}_{-m_3}.
\end{eqnarray*}
\end{cor}
%
%
\begin{prop}\label{Prop hat ad^*_e}
For ${(e_{l_km_k},a_k)}\in \hat{\cg}$, $k\in\{1,2\}$ we have
\begin{equation}
\hat{ad}^*_{(e_{l_2m_2},a_2)}(e_{l_1m_1},a_1)=\Big(  A^{l_3m_3}_{l_1m_1a_1 l_2m_2a_2}e_{l_3m_3},0\Big)
\end{equation}
where
\begin{equation}\label{A eq}
A^{l_3m_3}_{l_1m_1a_1 l_2m_2a_2}=-\frac{ \alpha^2+l_1(l_1+1) }{\alpha^2+l_3(l_3+1)}G^{l_3m_3}_{l_1m_1l_2m_2}   -  a_1  D^{l_3m_3}_{1~0~l_2m_2}.
\end{equation}
\end{prop}
\begin{proof}
We know that $\hat{ad}^*_{(e_{l_2m_2},a_2)}(e_{l_1m_1},a_1)  =  \Big(   {ad}^*_{e_{l_2m_2}}e_{l_1m_1} - a_1Te_{l_2m_2}  ,   0  \Big)$ where the operator $T$ is given by (\ref{eq Te^m_l}). Thus it suffices to compute ${ad}^*_{e_{l_2m_2}}e_{l_1m_1}$. Suppose that
\[{ad}^*_{e_{l_2m_2}}e_{l_1m_1}=a^{l_3m_3}_{l_1m_1l_2m_2}e_{l_3m_3}.\]
 We claim that
\[a^{l_3m_3}_{l_1m_1l_2m_2} = -\frac{\alpha^2 + l_1(l_1+1) }{\alpha^2+l_3(l_3+1)}G^{l_3m_3}_{l_1m_1l_2m_2}.
\]
In fact we have
\begin{eqnarray*}
\ll a^{l_3m_3}_{l_1m_1l_2m_2}e_{l_3m_3}  ,  e_{lm}   \gg_\cg  &=&  \ll {ad}^*_{e_{l_2m_2}}e_{l_1m_1}  ,  e_{lm}   \gg_\cg\\
&=&  \ll e_{l_1m_1}  ,  {ad}_{e_{l_2m_2}}e_{lm}   \gg_\cg\\
&=& - \ll e_{l_1m_1}  ,  [e_{l_2m_2}  ,  e_{lm}]   \gg_\cg\\
&=& - \ll e_{l_1m_1}  ,  G^{l_4m_4}_{l_2m_2lm}e_{l_4m_4}    \gg_\cg.
\end{eqnarray*}
Moreover, the first term of the above equality implies that
\begin{eqnarray*}
\ll a^{l_3m_3}_{l_1m_1l_2m_2}e_{l_3m_3}  ,  e_{lm}   \gg_\cg  &=&   a^{l_3m_3}_{l_1m_1l_2m_2}  \ll e_{l_3m_3}  ,  e_{lm}   \gg_\cg\\
&=&  a^{l_3m_3}_{l_1m_1l_2m_2} (-1)^{m_3} (\alpha^2 +l_3(l_3+1)) \delta^{l}_{l_3}\delta^{m}_{-m_3}.
\end{eqnarray*}
On the other hand
\begin{eqnarray*}
- \ll e_{l_1m_1}  ,  G^{l_4m_4}_{l_2m_2lm}e_{l_4m_4}    \gg_\cg     &=&    - G^{l_4m_4}_{l_2m_2lm}(\alpha^2 +l_1(l_1+1))  \langle Y_{l_1m_1}, Y_{l_4m_4}  \rangle\\
&=&  - (-1)^{m_1} G^{l_4m_4}_{l_2m_2lm}(\alpha^2 +l_1(l_1+1))  \delta^{l_1}_{l_4}  \delta^{m_1}_{-m_4}\\
&=& - (-1)^{m_1} (\alpha^2 +l_1(l_1+1))  G^{l_1-m_1}_{l_2m_2lm}   \\
&=& - (-1)^{m_1}  (\alpha^2 +l_1(l_1+1))  G^{l_1-m_1}_{l_2m_2l_3-m_3}   \\
&=& - (-1)^{m_3}  (\alpha^2 +l_1(l_1+1))  G^{l_3m_3}_{l_1m_1l_2m_2}.
\end{eqnarray*}
For the last line of the above equality we used equations \eqref{real structure constants} and \eqref{eq symmetry}.
The above two equations imply that
\begin{equation}
a^{l_3m_3}_{l_1m_1l_2m_2} = -\frac{   \alpha^2 + l_1(l_1+1) }{\alpha^2 +l_3(l_3+1)}G^{l_3m_3}_{l_1m_1l_2m_2}.
\end{equation}
Finally
\begin{eqnarray*}
\hat{ad}^*_{(e_{l_2m_2},a_2)}(e_{l_1m_1},a_1)  &=&  \Big(   {ad}^*_{e_{l_2m_2}}e_{l_1m_1} - a_1Te_{l_2m_2}  ,   0  \Big)\\
&=&    \Big(   [-\frac{   \alpha^2 + l_1(l_1+1) }{\alpha^2 +l_3(l_3+1)}G^{l_3m_3}_{l_1m_1l_2m_2}      -  a_1D^{l_3m_3}_{1~0~l_2m_2}]e_{l_3m_3}  ,   0  \Big)\\
&:=&  \Big(   A^{l_3m_3}_{l_1m_1a_1l_2m_2a_2}  e_{l_3m_3}  ,   0  \Big)
\end{eqnarray*}
\end{proof}
%
%
%
\begin{prop}\label{Prop hat nabla e}
For $u=e_{l_1m_1}$ and $v=e_{l_2m_2}$ we have
\begin{eqnarray}\label{eqn hat Nabla e}
\hat\nabla_{(u,a_1)}(v,a_2)    &=&   \Big(   \Gamma^{l_3m_3}_{l_1m_1a_1l_2m_2a_2}e_{l_3m_3}     ,     -i\frac{m_1}{2}(-1)^{m_1}\delta^{l_1}_{l_2}\delta^{m_1}_{-m_2}   \Big)
\end{eqnarray}
where
\begin{equation}\label{Gamma eq }
\Gamma^{l_3 m_3 }_{l_1m_1a_1l_2m_2a_2}   =   d_{l_1l_2}^{l_3 }  G^{l_3 m_3 }_{l_1m_1l_2m_2}  -  \frac{1}{2}(a_1 D^{l_3 m_3 }_{1~0~l_2m_2}  +a_2 D^{l_3 m_3 }_{1~0~l_1m_1})
\end{equation}
and
\begin{equation*}
d^{l_3}_{l_1l_2}=\frac{1}{2}\frac{\alpha^2  + l_3(l_3+1)   -l_1(l_1+1)  +l_2(l_2+1)}{     \alpha^2   + l_3(l_3+1)}.
\end{equation*}
\end{prop}
\begin{proof}
Using lemma (\ref{lemma  Te}), corollary (\ref{Cor Omega(e,e)}) and proposition (\ref{Prop hat ad^*_e}) we calculate
\begin{eqnarray*}
&&  2\hat\nabla_{(u,a_1)}(v,a_2)  =  - \hat{ad}_{(u,a_1)}(v,a_2)+\hat{ad}^*_{(u,a_1)}(v,a_2)+\hat{ad}^*_{(v,a_2)}(u,a_1)\\
&=&   \Big(    [u,v]     + ad^*_uv  -a_2Tu+ad^*_vu  -a_1Tv   ,    \Omega(u,v) \Big)   \\
&=&     \Big(       \big[    G^{l_3m_3}_{l_1m_1l_2m_2}   -   \frac{\alpha^2 + l_2(l_2+1)}{\alpha^2 + l_3(l_3+1)} G^{l_3m_3}_{l_2m_2l_1m_1}
-   \frac{\alpha^2 +  l_1(l_1+1)}{\alpha^2 + l_3(l_3+1)} G^{l_3m_3}_{l_1m_1l_2m_2}    \\
&&  -   (  a_2D^{l_3m_3}_{1~0~l_1m_1 }  + a_1  D^{l_3m_3}_{1~0~l_2m_2 }  )       \big] e_{l_3m_3}   ~,~      -im_1(-1)^{m_1}\delta^{l_1}_{ l_2} \delta^{m_1}_{-m_2}  \Big)\\
&=&     \Big(       \big[    2d^{l_3}_{l_1l_2} G^{l_3m_3}_{l_1m_1l_2m_2}      -  (a_1 D^{l_3m_3}_{1~0~l_2m_2}  + a_2 D^{l_3m_3}_{1~0~l_1m_1})  \big]e_{l_3m_3}   \\
&&      ,      -im_1(-1)^{m_1}\delta^{l_1}_{ l_2} \delta^{m_1}_{-m_2}  \Big)\\
\end{eqnarray*}
which completes the proof.
\end{proof}
%
%
%

We use a well-known result to derive an alternative expression for the curvature of right-invariant metrics. This equivalent form will be applied in theorem \ref{thm curv e}.\\
The primary motivation for using this form is to ensure compatibility with the results presented in \cite{Arak-Sav} and \cite{Dowker-Zheng}. However, it's worth noting that one could also employ Lukatskii's method \cite{Luk}, which would yield the same result, as demonstrated in section \ref{section sect. comp}.
\begin{Remark}\label{simplified curvature}
For the right (left) invariant vector fields $X,Y,Z$ on any manifold and the Levi-Civita connection $\nabla$ we have
\begin{equation*}
2\ll \nabla_XY,Z\gg  =  \ll [X,Y]  ,  Z\gg    -   \ll [Y,Z]  ,  X\gg  +  \ll [Z,X]  ,  Y\gg.
\end{equation*}
Moreover we have
\begin{eqnarray*}
&&\ll  R(X,Y)Z  ,   W  \gg  =   \ll  \nabla_X\nabla_YZ  -  \nabla_Y\nabla_XZ  -  \nabla_{[X,Y]}Z   ,   W  \gg\\
&=&  - \ll  \nabla_YZ    ,   \nabla_XW  \gg    +   \ll\nabla_XZ    ,   \nabla_YW \gg        -    \ll  \nabla_{[X,Y]}Z  , W\gg  \\
&=&  - \ll  \nabla_YZ    ,   \nabla_XW  \gg    +   \ll\nabla_XZ    ,   \nabla_YW \gg      \\
&& - \frac{1}{2}\{   \ll   [[X,Y]  ,  Z] , W\gg    -   \ll [Z,W]  ,  [X,Y] \gg  +  \ll [W,[X,Y]]  ,  Z\gg   \}    \\
&=&     \ll\nabla_XZ    ,   \nabla_YW \gg   - \ll  \nabla_YZ    ,   \nabla_XW  \gg     \\
&& + \frac{1}{2}\ll [X,Y]   ,  [Z,W]   - ad^*_ZW   +  ad^*_WZ    \gg.
\end{eqnarray*}
\end{Remark}
%
%
%
%
The following theorem, which is one of the main objectives of this paper, allows us to express the curvature based on the structure constant. The proof of this theorem has been moved to the appendix.
\begin{thm}\label{thm curv e}
The curvature of operator $\hat{R}$ on $\hat\cg$ is given by
\begin{eqnarray}\label{eq Curvature formula e}
&& \ll \hat{R}  \big(  ((e_{l_1m_1},a_1)  ,  (e_{l_2m_2},a_2)  \big) (e_{l_3m_3},a_3)   ,   (e_{l_4m_4},a_4)  \gg_{\hat\cg}\\
\nonumber &=& \sum_{m,l} (-1)^m(\alpha^2 + l(l+1)) \Big[ \Gamma^{lm}_{l_1m_1a_1l_3m_3a_3} \Gamma^{l-m}_{l_2m_2a_2l_4m_4a_4}    -
\Gamma^{lm}_{l_2m_2a_2l_3m_3a_3} \Gamma^{l-m}_{l_1m_1a_1l_4m_4a_4}\\
\nonumber &&+ G^{lm}_{l_1m_1 l_2m_2}\Big( G^{l-m}_{l_3m_3l_4m_4}k^l_{l_3l_4}    +   \frac{1}{2} \big(  a_4D^{l-m}_{1~0~l_3m_3}  -a_3D^{l-m}_{1~0~l_4m_4} \big)             \Big)    \Big]\\
\nonumber && - \frac{m_1m_2}{4} (-1)^{m_1+m_2} \big[ \delta^{l_1}_{ l_3} \delta^{m_1}_{-m_3}    \delta^{l_2}_{l_4}\delta^{m_2}_{-m_4}
- \delta^{l_2}_{ l_3} \delta^{m_2}_{-m_3}    \delta^{l_1}_{ l_4} \delta^{m_1}_{-m_4}  \big]\\
\nonumber  &&-\frac{m_1m_3}{2} (-1)^{m_1+m_3}    \delta^{l_1}_{ l_2} \delta^{m_1}_{-m_2}    \delta^{l_3}_{ l_4} \delta^{m_3}_{-m_4}
\end{eqnarray}
where
\begin{equation}\label{k l}
k^l_{l_3l_4}=\frac{1}{2}\frac{    -\alpha^2  +  l(l+1)  -l_3(l_3+1)  - l_4(l_4+1)   }{   \alpha^2+l(l+1)   }.
\end{equation}
\end{thm}
%
%
%
%
%
%
%
\begin{Remark}
In the case that $\alpha=a_1=a_2=0$ then
\[\hat{ad}^*_{(e_{l_2m_2},0)}(e_{l_1m_1},0)={ad}^*_{e_{l_2m_2}}e_{l_1m_1}=-B(e_{l_1m_1},e_{l_2m_2})=-\frac{ l_1(l_1+1) }{l_3(l_3+1)}G^{l_3m_3}_{l_1m_1l_2m_2}e_{l_3m_3}
\]
reduces to  equation (15) of \cite{Arak-Sav}. Note that the operator $B$ is \cite{Arak-Sav} is given by
\begin{equation*}
\ll B(X,Y),Z\gg_\cg=\ll X, [Y,Z] \gg_\cg
\end{equation*}
or equivalently
\begin{equation}\label{operator B}
ad^*_YX=-B(X,Y).
\end{equation}
More precisely we have
\begin{equation*}
\ll X, [Y,Z] \gg_\cg= -\ll X, ad_YZ \gg_\cg = -\ll ad^*_YX, Z \gg_\cg = \ll B(X,Y),Z\gg_\cg
\end{equation*}
for any $Z\in \cg$ which implies that $B(X,Y)=- ad^*_YX$.
Moreover the  expression \eqref{eqn hat Nabla e} for the covariant derivative coincides with equation (20) of \cite{Arak-Sav}.
Finally if  $a_i=0$ for $1\leq i\leq 4$ then, the curvature formula (\ref{eq Curvature formula e}) reduces to formula (23) of \cite{Arak-Sav}.
\end{Remark}
%
%
%
%
%
%
\begin{Remark}\label{Effect of the sign in Poisson bracket on curvature }
If the Poisson bracket is  defined  by
\begin{equation}\label{Poisson bracket skiba}
\{f,g\}_P:=  \frac{\partial f}{\partial \mu} \frac{\partial g}{\partial \lambda} -\frac{\partial f}{\partial \lambda}\frac{\partial g}{\partial \mu}
\end{equation}
(e.g. \cite{Lee-Pre} equation (8)) then $\{Y_{l_1m_1},Y_{l_2m_2}\}_P=-\{Y_{l_1m_1},Y_{l_2m_2}\}$. If we suppose that
\begin{eqnarray*}
\{Y_{l_1m_1},Y_{l_2m_2}\}_P:=\mathbb{G}^{l_3m_3}_{l_1m_1l_2m_2}Y_{l_3m_3}
\end{eqnarray*}
then 
\begin{equation*}
\mathbb{G}^{l_3m_3}_{l_1m_1l_2m_2} = -G^{l_3m_3}_{l_1m_1l_2m_2}\quad ,\quad \mathbb{D}^{l_3m_3}_{l_1m_1 l_2 m_2}=-D^{l_3m_3}_{l_1m_1l_2 m_2}
\end{equation*}
Since in the formula (\ref{eq Curvature formula e}) the terms include multiplication of two of the above terms, the sign convention for the Poisson bracket would not change the result of theorem \ref{thm curv e}.
\end{Remark}
%
%
%
%
%
%
%
%
\section{Sectional curvature computations}\label{section sect. comp}
\subsection{Curvature  of the zonal flow $Y_{10}$}\label{Cur of zonal flow Y10}
Equation (\ref{eq Curvature formula e})  enables us to compute the sectional curvature. In this part we consider the assumptions of Dowker and Mo-zheng \cite{Dowker-Zheng} and extend them to the case of rotating sphere. More precisely  suppose that $\hat\xi=(\nu\sqrt{\frac{4\pi}{3}}e_{10},a)$ be a zonal flow. Moreover, let $\hat{\eta}=(\eta,b)\in \hat\cg$ be a unit vector field such that it contains only one angular momentum $l=l_0$ in the form
$\hat{\eta}=(\eta_{l_0m_0}e_{l_0m_0} + \eta_{l_0-m_0}e_{l_0-m_0} ,0)$ with the assumption
\begin{eqnarray}\label{norm-eta=1}
\ll \hat{\eta},\hat{\eta}\gg_{\hat{\cg}}  &=& 2 \eta_{l_0m_0}\eta_{l_0-m_0}\ll e_{l_0m_0} , e_{l_0-m_0}\gg_\cg\\
&=& 2(-1)^{m_0}(\alpha^2+l_0(l_0+1))\eta_{l_0m_0}\eta_{l_0-m_0}=1 \nonumber
\end{eqnarray}
Then we have
\begin{eqnarray}\label{sect. cur Y 1 0}
\hat{\kappa}(\hat\xi,\hat\eta)  &=&   \frac{\ll  \hat{R}\big(\hat\xi,\hat\eta\big)\hat\eta~ ,~\hat\xi    \gg_{\hat\cg}}{\ll\hat\xi,\hat\xi\gg_{\hat\cg}}\\
\nonumber&=&   \frac{    m_0^2    }{    (\nu^2     \frac{4\pi  }{3}(\alpha^2+2)+  a^2)    (\alpha^2+l_0(l_0+1))^2      }             \Big[    \frac{  \nu^2 }{ 4}(\alpha^2+2)^2     +      \frac{\nu a}{2}   \left( \alpha^2+2 \right)   +     \frac{ a^2}{4}    \Big]
\end{eqnarray}
For  the details of computations see appendix. Note that for the case that $a=\alpha=0$ we have
\begin{eqnarray*}
\hat{\kappa}(\hat\xi,\hat\eta)  &=&
\frac{3}{8\pi}  \frac{    m_0^2    }{          l_0^2(l_0+1)^2      }
\end{eqnarray*}
which coincides with the equation (9) from \cite{Dowker-Zheng}.
%

%
%
%
%
%
%
\subsection{Sectional curvature and stability of the tradewind current}
Arnold in \cite{Arnold} assumed that the atmosphere is a two dimensional torus and the motion of the atmosphere is approximated by the tradewind $\xi=(\sin y,0)$ current and he computed the sectional curvature in the planes containing the tradewind current velocity field. He noticed that the sectional curvature in many directions is negative which makes the prediction of weather for long time  practically impossible.  Lukatskii in \cite{Luk} computed the sectional curvature of the tradewind flow (named by Lukatskii the "passat flow") and he extended the results of Arnold in  the more realistic case of two dimensional sphere. His results approved the estimates  of Arnold for the unreliability  of the weather forecast in longtime.\\
In this part  using theorem  \ref{thm curv e}, we compute the sectional curvature of the tradewind current at the presence of the Coriolis effect  which makes the situation more practical.

Consider the stream function $g(\mu)=\frac{1}{2}\sqrt{\frac{15}{8\pi}} \mu^2$ where the coefficient, as it is introduced in \cite{Luk}, is for normalization.

Since $\{ g,Y_{lm} \}=-\partial_\mu g\partial_\lambda Y_{lm}$, using the recursive relation (see e.g. \cite{Mess}, p. 493, eq. B.78)
\begin{equation}\label{rec eq for tradewind str cts}
\mu P^m_l(\mu) = \frac{l+m}{2l+1} P^m_{l-1}   +  \frac{l-m+1}{2l+1} P^m_{l+1}
\end{equation}
we get
\begin{eqnarray}\label{structure constants for tradewind}
\{ g,Y_{lm} \}=   -  im \sqrt{\frac{15}{8\pi}} \Big(  \sqrt{a^l_m} Y_{l-1 m}  +  \sqrt{a^{l+1}_m} Y_{l+1 m}  \Big)
\end{eqnarray}
where
$$a^l_m=\frac{l^2-m^2}{4l^2-1}.$$
We address the structure constants of this stream function (and consequently the corresponding vector field) by $G^{lm}_{\widehat{2~0} ~l_0m_0}$. More precisely we set
\begin{eqnarray*}
\{ g,Y_{l_0m_0} \} := G^{lm}_{\widehat{2~0} ~l_0m_0}e_{lm}.
\end{eqnarray*}
Equation (\ref{structure constants for tradewind}) implies that the only non-vanishing structure constants of the tradewind flow are
\begin{equation}\label{str cts for tradewind l-1}
G^{l_0-1 m_0}_{\widehat{2~0} ~l_0m_0} =  - G^{l_0-1 -m_0}_{\widehat{2~0} ~l_0-m_0}  = -im_0 \sqrt{\frac{15}{8\pi}}\sqrt{a^{l_0}_{m_0}}
\end{equation}
and
\begin{equation}\label{str cts for tradewind l+1}
G^{l_0+1 m_0}_{\widehat{2~0} ~l_0m_0} = -G^{l_0+1 -m_0}_{\widehat{2~0} ~l_0-m_0}  = -im_0 \sqrt{\frac{15}{8\pi}}\sqrt{a^{l_0+1}_{m_0}}.
\end{equation}
%
%
%
%
\begin{rem}
Note that $g(\mu) = \frac{1}{2} \sqrt{\frac{15}{8\pi}} \mu^2$ and $Y_{2~0} = \sqrt{\frac{5}{16\pi}} (3\mu^2 - 1)$ have distinct properties. For instance, they exhibit different structure constants, resulting in diverse sectional curvature (for the real structure constants of $Y_{2~0}$, refer to \cite{Arak-Sav}, equation (A20)). As shown in figure \ref{fig:boat1}, the sectional curvature of $g$ and the unit vectors $\eta = \eta_{l_0m_0}e_{l_0m_0} + \eta_{l_0-m_0}e_{l_0-m_0}$, where $|m_0| > 1$, is negative (see also \cite{Luk}, theorem 2, part). On the other hand, the sectional curvature $\kappa(\sgrad Y_{2~0}, \eta)$ is positive, as demonstrated in \cite{Dowker-Zheng}, page 2364.
\end{rem}
%
%
\begin{la}
For $e_{\widehat{20}}=\sgrad  g$ and $a_1,a_2\in\mathbb{R}$ and $m_0\in\mathbb{N}$ we have
\begin{eqnarray*}
&& \ll \hat{R}  \big(  ((\hat{e}_{20},a_1)  ,  (e_{l_0m_0},a_2)  \big) (\hat{e}_{20},a_1)   ,   (e_{l_0-m_0},a_2)  \gg_{\hat\cg}\\
&&=-\frac{15m_0^2}{32\pi}(-1)^{m_0}   \big(  \alpha^2+l_0(l_0+1)   \big)\Theta    -    \frac{(-1)^{m_0}}{4} \frac{a_1^2m_0^2}{  \alpha^2+l_0(l_0+1)   }
\end{eqnarray*}
where the term $\Theta$ is given by
\begin{eqnarray}\label{Theta}
\Theta  &=&  \Big[ (1-c_{l_0})^2  (  a_{m_0}^{l_0}b_{l_0} + \frac{a_{m_0}^{l_0+1}}{b_{l_0+1}}  )   +  2(1+c_{l_0}) (  a^{l_0}_{m_0} +  a^{l_0+1}_{m_0} )
-3 (\frac{a_{m_0}^{l_0}}{b_{l_0}}   +  a_{m_0}^{l_0+1} b_{l_0+1}  )\nonumber\\
&&  +  x_{l_0-1} \frac{a_{m_0}^{l_0}}{b_{l_0}}  \Big(  x_{l_0-1}  -2b_{l_0}(1-c_{l_0})  +2   \Big)  \nonumber\\
&&  +  x_{l_0+1}    a_{m_0}^{l_0+1} b_{l_0+1}  \Big(  x_{l_0+1}  -\frac{2}{b_{l_0+1}}(1-c_{l_0})  +2   \Big)  \Big]
\end{eqnarray}
and $a^l_m=\frac{l^2-m^2}{4l^2-1}$,  $b_l=\frac{\alpha^2+l(l+1)}{\alpha^2+l(l-1)}$, $c_l=\frac{6}{\alpha^2+l(l+1)}$ and $x_l=\frac{\alpha^2}{\alpha^2+l(l+1)}$.
\end{la}
%
%
\begin{proof}
Using the formula in theorem \ref{thm curv e} and the fact that $\Gamma^{lm}_{2~0~a_1~2~0~a_1}=D^{lm}_{1~0~2~0}=0$ for the term
\begin{equation*}
\mathcal{\hat{R}}:= \ll \hat{R}  \big(  (e_{\widehat{20}},a_1)  ,  (e_{l_0m_0},a_2)  \big) (e_{\widehat{20}},a_1)   ,   (e_{l_0-m_0},a_2)  \gg_{\hat\cg}
\end{equation*}
we have
\begin{eqnarray*}
\mathcal{\hat{R}} &=& \sum_{m,l} (-1)^m(\alpha^2 + l(l+1)) \Big[     - \Gamma^{lm}_{l_0m_0a_2\widehat{2~0}a_1} \Gamma^{l-m}_{\widehat{2~0}a_1l_0-m_0a_2}\\
&&+ G^{lm}_{\widehat{2~0} ~l_0m_0}\Big( G^{l-m}_{\widehat{2~0}~ l_0-m_0}k^l_{2l_0}    -  \frac{1}{2} a_1D^{l-m}_{1~0~l_0 -m_0}              \Big)    \Big]
\end{eqnarray*}
Due to the fact that in the expression of $G^{l\pm m}_{\widehat{2~0~} l_0\pm m_0}$  and $G^{l\pm m}_{1~0~l_0-m_0}$ contain the terms $\delta_{l_0\pm 1}^l$ and $\delta_{l_0}^l$ respectively, then their multiplication is always zero.
As a consequence we have
\begin{eqnarray*}
&&  \sum_{m,l} \Gamma^{lm}_{l_0m_0a_2\widehat{2~0}a_1} \Gamma^{l-m}_{\widehat{2~0}a_1l_0-m_0a_2}\\
&=& \sum_{m,l}    \Big(     d^l_{l_02}G^{lm}_{l_0m_0\widehat{2~0}}    -   \frac{1}{2} (a_2D^{lm}_{1~0~\widehat{2~0}}  +  a_1D^{lm}_{1~0~l_0 m_0}    \Big)    \\
&&\times \Big(     d^l_{2l_0}G^{l-m}_{\widehat{2~0}~l_0-m_0}    -    \frac{1}{2} (  a_1D^{l-m}_{1~0~ l_0 -m_0}  +  a_2D^{lm}_{~1~0~\widehat{2~0}}    \Big)\\
&=&   \sum_{m,l} d^l_{l_02}G^{lm}_{l_0m_0\widehat{2~0}}  d^l_{2l_0}G^{l-m}_{\widehat{2~0}~l_0-m_0}   +  \frac{1}{4}    a_1^2   D^{lm}_{1~0~l_0 m_0}  D^{l-m}_{1~0~ l_0 -m_0}\\
&=&   \sum_{m,l} d^l_{l_02}G^{lm}_{l_0m_0\widehat{2~0}}  d^l_{2l_0}G^{l-m}_{\widehat{2~0}~l_0-m_0}   +  \frac{1}{4}      \frac{  a_1^2 m_0^2  }{ (\alpha^2 +l_0(l_0+1))^2}  \delta_{m_0}^m\delta_{l_0}^l.
\end{eqnarray*}
Using the last equality  we get
\begin{eqnarray*}
\mathcal{\hat{R}} &=& \sum_{m,l} (-1)^m(\alpha^2 + l(l+1)) \Big[       -d^l_{l_02}d^l_{2l_0}   G^{lm}_{l_0m_0\widehat{2~0}}  G^{l-m}_{\widehat{2~0}~l_0-m_0}   \\
&&+ k^l_{2l_0}G^{lm}_{\widehat{2~0}~ l_0m_0} G^{l-m}_{\widehat{2~0}~ l_0-m_0}       -  \frac{1}{4}    \frac{  a_1^2 m_0^2  }{ (\alpha^2 +l_0(l_0+1))^2} \delta_{m_0}^m\delta_{l_0}^l \Big]\\
%
\end{eqnarray*}
On the other hand using the structure constants (\ref{str cts for tradewind l-1}), (\ref{str cts for tradewind l+1}) and an straightforward calculation  we observe that
\begin{eqnarray*}
&&\sum_{m,l} (-1)^m(\alpha^2 + l(l+1)) \Big[       d^l_{l_02}d^l_{2l_0}   G^{lm}_{\widehat{2~0}~l_0m_0}  G^{l-m}_{\widehat{2~0}~l_0-m_0}
+ k^l_{2l_0}G^{lm}_{\widehat{2~0} l_0m_0} G^{l-m}_{\widehat{2~0}~ l_0-m_0}       \Big]\\
&&\sum_{m,l} (-1)^m(\alpha^2 + l(l+1))    G^{lm}_{\widehat{2~0}~l_0m_0}  G^{l-m}_{\widehat{2~0}~l_0-m_0}   \Big[       d^l_{l_02}d^l_{2l_0}
+ k^l_{2l_0}          \Big]\\
&=& -\frac{15m_0^2}{8\pi}(-1)^{m_0}(\alpha^2 + (l_0-1)l_0)  a_{m_0}^{l_0}    \Big(  -d^{l_0-1}_{l_02} d^{l_0-1}_{2l_0}  - k^{l_0-1}_{2l_0}   \Big)\\
&&  -  \frac{15m_0^2}{8\pi}(-1)^{m_0}    \big(    \alpha^2 + (l_0+1)(l_0+2)    \big)     a_{m_0}^{l_0+1}    \Big(  -d^{l_0+1}_{l_02} d^{l_0+1}_{2l_0}  - k^{l_0+1}_{2l_0}   \Big)\\
&=&  - \frac{15m_0^2}{32\pi^2}(-1)^{m_0}   \big(  \alpha^2+l_0(l_0+1)   \big)\Theta.
\end{eqnarray*}
which completes the proof.
\end{proof}
%
%
%
%
Now, suppose that  $e_{\widehat{2~0}}=\sgrad \frac{1}{2}\sqrt{\frac{15}{8\pi}} \mu^2$. Moreover suppose that $\eta=\eta_{l_0m_0}e_{l_0m_0} + \eta_{l_0-m_0}e_{l_0-m_0}$ with $m_0\in\mathbb{N}$ be a vector with norm one. Consider the vectors $\hat{\xi}=(e_{\widehat{2~0}},a)$ and  $\hat{\eta}=(\eta,0)$ in $\hat{\cg}$. In the next step we calculate the sectional curvature of the plane which contains the tradewind current vector field $\hat\xi$ and $\hat\eta$.
%
%
%
\begin{thm}\label{theorem sectional curv of tw}
\textbf{i}. For the two dimensional planes containing the tradewind current and $\hat{\eta}$ we have
\begin{eqnarray*}
\hat{\kappa}(\hat{\xi},\hat\eta)= \frac{1}{\frac{3}{8}\alpha^2 + 1  + a^2}      \Big(    \frac{15m_0^2}{32\pi}\Theta  +  \frac{1}{4}\frac{a^2m_0^2}{(\alpha^2 +l_0(l_0+1))^2}   \Big).
\end{eqnarray*}
where $\Theta$ is given by (\ref{Theta}).\\
\textbf{ii}. For $\hat{\eta}=\left( \eta_{l_0l_0}e_{l_0l_0} + \eta_{l_0-l_0}e_{l_0-l_0}, 0 \right)$ we have
\begin{eqnarray*}
\lim_{l_0\rightarrow \infty}\hat{\kappa}(\hat{\xi},\hat\eta)= \frac{-1}{\frac{3}{8}\alpha^2 + 1  + a^2}         \frac{15}{8\pi}
\end{eqnarray*}
\end{thm}
\begin{proof}
\textbf{i}. The proof is a direct consequence of the previous lemma. In fact we have $\ll\hat{\xi},\hat{\xi}\gg_{\hat\cg}  =  \frac{3}{8}\alpha^2 + 1  + a^2$ and
\begin{eqnarray*}
\hat{\kappa}(\hat{\xi},\hat\eta)  &=&   \frac{  \ll  \hat{R}(\hat\xi,\hat\eta)\hat\eta ~,~\hat\xi  \gg_{\hat\cg}    }{
\ll  \hat\xi,\hat\xi\gg_{\hat\cg}    \ll  \hat\eta,\hat\eta\gg_{\hat\cg} - \ll  \hat\xi,\hat\eta\gg_{\hat\cg}^2  }=  - \frac{  \ll  \hat{R}(\hat\xi,\hat\eta)\hat\xi ~,~\hat\eta  \gg_{\hat\cg}    }{
\ll  \hat\xi,\hat\xi\gg_{\hat\cg}      }\\
&=& -\frac{ 2\ll \hat{R}  \big(  (\hat{e}_{20},a_1)  ,  (e_{l_0m_0},0)  \big) (\hat{e}_{20},a_1)   ,   (e_{l_0-m_0},0)  \gg_{\hat\cg}  \eta_{l_0m_0}\eta_{l_0-m_0}}{
\ll  \hat\xi,\hat\xi\gg_{\hat\cg}      }\\
&=& \frac{-2(-1)^{m_0}\eta_{l_0m_0}\eta_{l_0-m_0}}{  \frac{3}{8}\alpha^2 + 1  + a^2   }     \Big(    - \frac{15m_0^2}{32\pi}(\alpha^2+l_0(l_0+1))\Theta  \\
&&-\frac{1}{4}\frac{a^2m_0^2}{(\alpha^2 +l_0(l_0+1))}   \Big)  \\
&=& \frac{1}{  \frac{3}{8}\alpha^2 + 1  + a^2   }     \Big(     \frac{15m_0^2}{32\pi}\Theta  +  \frac{1}{4}\frac{a^2m_0^2}{(\alpha^2 +l_0(l_0+1))^2}   \Big)  \\
\end{eqnarray*}
where in the last line we used the equation (\ref{norm-eta=1}) from appendix.\\[2mm]
\textbf{ii}. To prove the second part we note that $\lim_{l_0\rightarrow\infty}\frac{l_0^2}{(\alpha^2 +l_0(l_0+1))^2}=0$ and $\lim_{l_0\rightarrow\infty}l_0^2\Theta=-4$.
\end{proof}
%
%
\begin{Remark}
Figure \ref{fig:1}, represents the sectional curvature $\hat{\kappa}(\hat{\xi},\hat\eta)$ with $\hat{\xi}=(e_{\widehat{2~0}},a)$ and $\hat{\eta}=\left( \eta_{l_0l_0}e_{l_0l_0} + \eta_{l_0-l_0}e_{l_0-l_0}, 0 \right)$ as we discussed in theorem \ref{theorem sectional curv of tw}, part ii.
As we can see from the diagram \ref{fig:1},  bigger  parameters of $"a"$  makes the curvature of tradewind closer to zero and even for some $l_0$ the curvature will be positive. For example for  $a=12$ the wave numbers  $l_0=2,3$  result positive sectional curvature.
\end{Remark}
\begin{figure}[!h]
\centering
  \includegraphics[width=7cm]{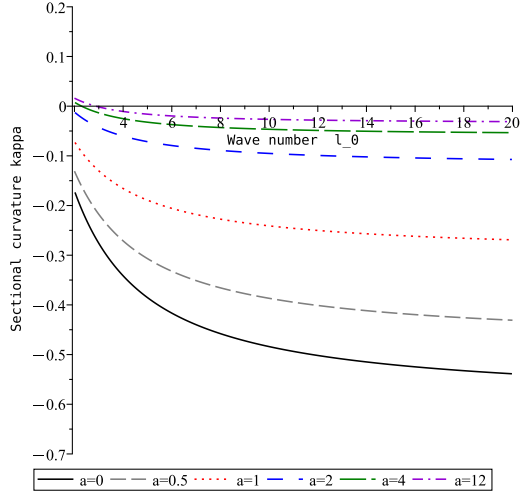}
  \caption{Sectional curvature of the tradewind current}\label{fig:1}
  \label{fig:boat1}
\end{figure}
%
%
%
%
%
%
\begin{Remark}
If $a=\alpha=0$ then $\hat{\kappa}(\hat{\xi},\hat\eta)=  \frac{15m_0^2}{32\pi}\Theta$ is the result in part 1, theorem 2 of \cite{Luk}.
However, note that Arnold in \cite{Arnold} and Lokatskii in \cite{Luk} used the inner products of the form
\begin{eqnarray*}
<<  e_{l_1m_1}  ,   e_{l_2m_2}    >>    :=  \langle \Delta_\theta(  Y_{l_1m_1}  ) ,  \overline{Y}_{l_2m_2}    \rangle
=\int_{S^2} (\alpha^2-\Delta)Y_{l_1m_1} \overline{Y}_{l_2m_2} d\mu
\end{eqnarray*}
which is different by a complex conjugate from our consideration (\ref{inner pproduct}).
\end{Remark}
%
%
%
%
%
%
\subsection{Unreliability of long-term weather forecast}
Following the discussions of \cite{Arnold, Luk}, we assume that the atmosphere is modeled as a homogeneous incompressible fluid on $\mathbb{S}^2$.
Taking into account the notations presented in  \cite{Luk}, we consider the
coordinate system
\begin{eqnarray*}
f:(-1,1)\times (0,2\pi)  &\longrightarrow & \Sph^2\subseteq\mathbb{R}^3\\
(\mu,\lambda)   &\longmapsto & (\sqrt{1-\mu^2}\sin\lambda  ,  \sqrt{1-\mu^2}\cos\lambda  ,  \mu)
\end{eqnarray*}
for $\Sph^2$.
Then the associated induced metric on $\Sph^2$ has the coefficients $g_{11}=\frac{1}{1-\mu^2}$, $g_{22}=1-\mu^2$ and $g_{12}=g_{21}=0$. The skew-gradient of the stream function $g(\mu)=\frac{1}{2}\sqrt{\frac{15}{8\pi}}\mu^2$ is
\begin{equation*}
e_{\widehat{2~0}}(\mu,\lambda)=-\sqrt{  \frac{15}{8\pi}   }   \mu\partial_\lambda
\end{equation*}
Using chain rule and the above mentioned coordinate system one can write $e_{\widehat{2~0}}(\mu,\lambda) =  \sqrt{  \frac{15}{8\pi}  } \mu(-y\partial_x+x\partial_y)$ which is the alternative form used in \cite{Luk}.
The maximus velocity of particles in this vector field is obtained by finding the critical points of $\|e_{\widehat{2~0}} \|^2={\frac{15}{8\pi}} \mu^2(1-\mu^2)$. Zero  is a minimum and $\mu=\pm\frac{1}{\sqrt{2}}$ are the maximums with the value $\|e_{\widehat{2~0}}(\pm\frac{1}{\sqrt{2}} , \lambda) \|=\sqrt{\frac{15}{8\pi}}\frac{1}{2}$. In the other words, the fastest particles with the  maximus velocity $\sqrt{\frac{15}{8\pi}}\frac{1}{2}$ are located on the latitudes $\pm\frac{\pi}{4}$.

Now, we suppose that the motion of the atmosphere is approximated by the tradewind current $\hat{{\xi}}=(e_{\widehat{2~0}} , a)  \in  \hat\cg$. In fact, the parameter $"a"$ from central extension, represents the rotation effect in the solution.
Since in the large scale $\alpha\ll 1$ we consider the case that $\alpha=0$. Of course for the case that $\alpha\neq0$ the same argument works with considering the Froude number in the sectional curvature.

As a result of theorem \ref{theorem sectional curv of tw}, we define the \emph{mean sectional curvature} to be
$$\hat{\kappa}_{av}=\frac{-15}{8\pi( 1+ a^2)}\beta_{Luk}$$
where $\beta_{Luk}\in(0,1)$ is a constant suggested by Lukatskii  in \cite{Luk}.  He applied the parameter $\beta_{Luk}$ and   measured the  error of prediction for $\beta_{Luk}=\frac{1}{4}$ and $\beta_{Luk}=\frac{1}{16}$. However, he used the symbol $\alpha$ for this constant which du to its similarity with the Froude number we replaced it with $\beta_{Luk}$.\\

As it is stated in \cite{Arnold-Khesin}, chapter 4 part 1C, if the sectional curvature in (almost) all directions in negative, then the distance of initially (infinitely) close geodesics increases exponentially by the equation  $y(s)=e^{\sqrt{-\kappa_{av}}s}$.
Because of that we discuss the exponential stability/instability of quasi-geostrophic motions  which is different from  the stability of the corresponding vector fields (see also \cite{Arnold-M}, Appendix 2).\\
Recall that the characteristic path length is the average path length  which an small error in the initial condition is increased by a factor of $e$ (\cite{Arnold-Khesin}, Remark 1.16).

Following \cite{Arnold-Khesin, Luk} we set the average path length to be
\begin{equation*}
S:=\sqrt{-\frac{1}{\hat{\kappa}_{av}}} =   \sqrt{\frac   {8\pi( 1+ a^2)}  {15 \beta_{Luk}}  }
\end{equation*}
Note that the parameter $a$ appears in $\kappa_{av}$ and the maximum velocity of particles of the tradewind  on $\Sph^2$ remains unchanged as it is computed above. More precisely, the fastest particle at time $T:= \sqrt{\frac   {8\pi( 1+ a^2)}  {15 \beta_{Luk}}  }$ travels the distance
$\frac{1}{2}\sqrt{\frac   {   1+ a^2 }  { \beta_{Luk}}  }$. Since on $\pm45^\circ$ the orbit length is $\sqrt{2}\pi$, the necessary time for this particle to travel one orbit is
\begin{equation*}
s: = 2\sqrt{2}\pi    \sqrt{\frac   { \beta_{Luk}}  {   1+ a^2 }    }  \sqrt{\frac   {8\pi( 1+ a^2)}  {15 \beta_{Luk}}  }
=2\sqrt{2}\pi    \sqrt{\frac   { \beta_{Luk}}  {   1+ a^2 }    }  \sqrt{\frac{-1}{\hat{\kappa}_{av}  }}.
\end{equation*}
The error $\epsilon$ after the necessary time for one orbit (one day) increases with the factor
\begin{eqnarray*}
\epsilon e^{\sqrt{-\hat{\kappa}_{av}} s}  =  \epsilon e^{  \sqrt{ -\hat{\kappa}_{av}  } 2\sqrt{2}\pi    \sqrt{\frac     { \beta_{Luk}}{   1+ a^2 }  } \sqrt{\frac{-1}{\hat{\kappa}_{av}  }} }
=  \epsilon e^{   2\sqrt{2}\pi    \sqrt{\frac{ \beta_{Luk} }{   1+ a^2 }}    }
\end{eqnarray*}
The diameter of earth is approximately $12750 ~\textrm{km}$ and its equatorial circumference is approximately $40,000 ~\textrm{km}$. If we consider the average speed of the trade wind $100 ~\textrm{km/h}$ then the necessary time to pass the orbit at latitude $\pm 45^\circ$ is approximately  $ \frac{400}{\sqrt{2}}\simeq 282$ hours (less than 12 days).
Then after n months the initial error $\epsilon$ increases by the factor
\begin{eqnarray*}
\epsilon e^{  n \frac{30\times 24}{400} 4\pi    \sqrt{\frac{ \beta_{Luk} }{   1+ a^2 }}    }  =  \epsilon 10^{  n \frac{30\times 24}{400} 4\pi    \sqrt{\frac{ \beta_{Luk} }{   1+ a^2 }} \log^e_{10}   }   \simeq  \epsilon 10^{  10 n     \sqrt{\frac{ \beta_{Luk} }{   1+ a^2 }}     }
\end{eqnarray*}
If $a=0$ then, the results coincide with the estimates of \cite{Luk}. As we can see, the Coriolis effect decreases the exponent and we can say that rotation has stability effect. The same holds true for the Froude number $\alpha$. As an special case, inspired by \cite{Skiba}, page 46, equation 3.1.1, in the absence of friction and turbulent  we fix $a=2$. If we assume $\beta_{Luk}=\frac{1}{16}$, then the period which the error appears with the magnitude $10^5$ is $4.5$ months more than two times of the estimated time in \cite{Luk}. For $\beta_{Luk}=\frac{1}{4}$ after almost two months the initial error $\epsilon$ increases by the factor $\epsilon10^{4.5}$ as Arnold in \cite{Arnold} estimated. Clearly, the choice of average sectional curvature (and consequently $\beta_{Luk}$) is essential for the estimate. Yoshida in \cite{Yoshida} suggested another approach to determine the average sectional curvature.  \\

%
%
%
%
%
%
%
%
\section{Appendix}
%
%
Proof of theorem \ref{thm curv e}.\\
\begin{proof}
Using remark (\ref{simplified curvature}) we get
\begin{eqnarray*}
&& \ll \hat{R}  \big(  (e_{l_1m_1},a_1)  ,  (e_{l_2m_2},a_2)  \big) (e_{l_3m_3},a_3)   ,   (e_{l_4m_4},a_4)  \gg_{\hat\cg}\\
&=&\ll    \hat\nabla_{(e_{l_1m_1},a_1)}   (e_{l_3m_3},a_3)   ,  \hat\nabla_{(e_{l_2m_2},a_2)}(e_{l_4m_4},a_4)  \gg_{\hat\cg}\\
&&-  \ll    \hat\nabla_{(e_{l_2m_2},a_2)}   (e_{l_3m_3},a_3)   ,  \hat\nabla_{(e_{l_1m_1},a_1)}(e_{l_4m_4},a_4)  \gg_{\hat\cg}\\
&&+ \frac{1}{2}  \ll    [(e_{l_1m_1},a_1)  ,  (e_{l_2m_2},a_2)]     ~~,~~  [(e_{l_3m_3},a_3)   ,   (e_{l_4m_4},a_4)]  \\
&&-\hat{ad}^*_{(e_{l_3m_3},a_3)}    ,   (e_{l_4m_4},a_4)    +\hat{ad}^*_{(e_{l_4m_4},a_4)}(e_{l_3m_3},a_3)      \gg_{\hat\cg}.
\end{eqnarray*}
Using proposition (\ref{Prop hat nabla e}) the first term of the above expression becomes
\begin{eqnarray*}
&&   \ll    \hat\nabla_{(e_{l_1m_1},a_1)}   (e_{l_3m_3},a_3)   ,  \hat\nabla_{(e_{l_2m_2},a_2)}(e_{l_4m_4},a_4)  \gg_{\hat\cg}\\
&=&  \ll  \Big(   \Gamma^{l m }_{l_1m_1a_1l_3m_3a_3}e_{l m }    ~ ,~     -i\frac{m_1}{2}(-1)^{m_1}\delta^{l_1}_{ l_3} \delta^{m_1}_{-m_3}   \Big)   \\
&&~~,~~ \Big(   \Gamma^{l m }_{l_2m_2a_2l_4m_4a_4}e_{l m }    ~ ,~     -i\frac{m_2}{2}(-1)^{m_2}\delta^{l_2}_ {l_4} \delta^{m_2}_{-m_4}   \Big)   \gg_{\hat\cg}\\
&=& \sum_{m,l} (-1)^{m}(\alpha^2+l(l+1))   \Gamma^{l m }_{l_1m_1a_1l_3m_3a_3} \Gamma^{l -m }_{l_2m_2a_2l_4m_4a_4} \\
&&- \frac{m_1m_2}{4} (-1)^{m_1+m_2}  \delta^{l_1}_{ l_3} \delta^{m_1}_{-m_3}    \delta^{l_2}_{l_4} \delta^{m_2}_{-m_4}
\end{eqnarray*}
The second term is calculated similarly and it  is
\begin{eqnarray*}
&&\ll    \hat\nabla_{(e_{l_2m_2},a_2)}   (e_{l_3m_3},a_3)   ,  \hat\nabla_{(e_{l_1m_1},a_1)}(e_{l_4m_4},a_4)  \gg_{\hat\cg}\\
&=& \sum_{m,l} (-1)^m(\alpha^2 + l(l+1)) \Gamma^{lm}_{l_2m_2a_2l_3m_3a_3} \Gamma^{l-m}_{l_1m_1a_1l_4m_4a_4}\\
&& - \frac{m_1m_2}{4} (-1)^{m_1+m_2}   \delta^{l_2}_{ l_3} \delta^{m_2}_{-m_3}    \delta^{l_1}_{ l_4} \delta^{m_1}_{-m_4}  \big].
\end{eqnarray*}
Finally we have
\begin{eqnarray*}
&&  \frac{1}{2} \ll    [(e_{l_1m_1},a_1)  ,  (e_{l_2m_2},a_2)]     ~~,~~  [(e_{l_3m_3},a_3)   ,   (e_{l_4m_4},a_4)]  \\
&&-\hat{ad}^*_{(e_{l_3m_3},a_3)}    ,   (e_{l_4m_4},a_4)    +\hat{ad}^*_{(e_{l_4m_4},a_4)}(e_{l_3m_3},a_3)      \gg_{\hat\cg}\\
&=& \frac{1}{2}\ll    \Big( G^{lm}_{l_1m_1l_2m_2}e_{lm}   , \Omega(  e_{l_1m_1},e_{l_2m_2})  \Big)     ~~,~~  \Big( G^{lm}_{l_3m_3l_4m_4}e_{lm}   , \Omega(  e_{l_3m_3},e_{l_4m_4})  \Big)  \\
&&-  \Big( A^{lm}_{l_4m_4a_4l_3m_3a_3}e_{lm}   ,  0   \Big)  +  \Big( A^{lm}_{l_3m_3a_3l_4m_4a_4}e_{lm}   , 0  \Big)    \gg_{\hat\cg}\\
\end{eqnarray*}
\begin{eqnarray*}
&=&  \frac{1}{2}  \ll     G^{lm}_{l_1m_1l_2m_2}e_{lm}         ~~,~~    G^{lm}_{l_3m_3l_4m_4}e_{lm}     -   A^{lm}_{l_4m_4a_4l_3m_3a_3}e_{lm}  + A^{lm}_{l_3m_3a_3l_4m_4a_4}e_{lm}       \gg_{\cg}\\
&&    +  \frac{1}{2}\Omega(  e_{l_1m_1},e_{l_2m_2})\Omega(  e_{l_3m_3},e_{l_4m_4})  \\
&=& \frac{1}{2} \sum_{m,l} (-1)^m(\alpha^2 + l(l+1))   G^{lm}_{l_1m_1l_2m_2} \Big(G^{l-m}_{l_3m_3l_4m_4}     -   A^{l-m}_{l_4m_4a_4l_3m_3a_3} + A^{l-m}_{l_3m_3a_3l_4m_4a_4}\Big)\\
&&-\frac{m_1m_3}{2} (-1)^{m_1+m_3}    \delta^{l_1}_{ l_2} \delta^{m_1}_{-m_2}    \delta^{l_3}_{ l_4} \delta^{m_3}_{-m_4}.
\end{eqnarray*}
Using proposition  (\ref{Prop hat ad^*_e}) we obtain
\begin{eqnarray*}
&&  G^{l-m}_{l_3m_3l_4m_4}     -   A^{l-m}_{l_4m_4a_4l_3m_3a_3} + A^{l-m}_{l_3m_3a_3l_4m_4a_4}  \\
&=&   G^{l-m}_{l_3m_3l_4m_4}  - \Big( -\frac{\alpha^2+l_4(l_4+1)}{\alpha^2+l(l+1)} G^{l-m}_{l_4m_4 l_3m_3}   -    a_4D^{l-m}_{1~0~
l_3m_3}   \Big)\\
&& ~~~~~~~~~~~~~+  \Big( -\frac{\alpha^2+l_3(l_3+1)}{\alpha^2+l(l+1)} G^{l-m}_{l_3m_3 l_4m_4 }   -    a_3D^{l-m}_{1~0~
l_4m_4}   \Big)
\end{eqnarray*}
\begin{eqnarray*}
&=&   G^{l-m}_{l_3m_3l_4m_4} \Big(   1   -\frac{\alpha^2+l_4(l_4+1)}{\alpha^2+l(l+1)}  -\frac{\alpha^2+l_3(l_3+1)}{\alpha^2+l(l+1)}  \Big)   \\
&&+  \big( a_4D^{l-m}_{1~0~l_3m_3} - a_3D^{l-m}_{1~0~l_4m_4}  \big) \\
&=&   2 k^l_{l_3l_4} G^{l-m}_{l_3m_3l_4m_4}           +        \big( a_4D^{l-m}_{1~0~l_3m_3} - a_3D^{l-m}_{1~0~l_4m_4}\big)
\end{eqnarray*}
As a consequence we have
\begin{eqnarray*}
&&  \frac{1}{2} \ll    [(e_{l_1m_1},a_1)  ,  (e_{l_2m_2},a_2)]     ~~,~~  [(e_{l_3m_3},a_3)   ,   (e_{l_4m_4},a_4)]  \\
&&-\hat{ad}^*_{(e_{l_3m_3},a_3)}    ,   (e_{l_4m_4},a_4)    +\hat{ad}^*_{(e_{l_4m_4},a_4)}(e_{l_3m_3},a_3)      \gg_{\hat\cg}\\
&=&  \sum_{m,l} (-1)^m(\alpha^2 + l(l+1))   G^{lm}_{l_1m_1l_2m_2}     \Big(   G^{l-m}_{l_3m_3l_4m_4}      k^l_{l_3l_4}        +   \frac{ 1}{2}    \big( a_4D^{l-m}_{1~0~l_3m_3} - a_3D^{l-m}_{1~0~l_4m_4}\big)   \Big)\\
&&-\frac{m_1m_3}{2} (-1)^{m_1+m_3}     \delta^{l_1}_{ l_2} \delta^{m_1}_{-m_2}    \delta^{l_3}_{ l_4} \delta^{m_3}_{-m_4}.
\end{eqnarray*}
which completes the proof.
\end{proof}
%
%
%
%
In the sequel we present the \textbf{details for the  computations of equation \eqref{sect. cur Y 1 0} in  section \ref{Cur of zonal flow Y10}.}\\
In this part we consider the assumptions of Dowker and Mo-Zheng \cite{Dowker-Zheng} and extend them to the case of rotating sphere. More precisely  suppose that $\hat\xi=(\nu\sqrt{\frac{4\pi}{3}}e_{10},a)$ be a zonal flow. Moreover, let $\hat{\eta}=(\eta,b)\in \hat\cg$ be a unit vector field such that it contains only one angular momentum $l=l_0$ in the form $\eta=\eta_{l_0m_0}e_{l_0m_0} + \eta_{l_0-m_0}e_{l_0-m_0}$ with $m_0\neq 0$.   Then the sectional curvature of the central extension is
\begin{eqnarray*}
\hat{\kappa}(\hat\xi,\hat\eta)  &=&   \frac{\ll  \hat{R}\big(\hat\xi,\hat\eta\big)\hat\eta~ ,~\hat\xi    \gg_{\hat\cg}}{\ll\hat\xi,\hat\xi\gg_{\hat\cg}}\\
&=&   -  \frac{1}{    \frac{4\pi \nu^2 }{3}(\alpha^2+2)+ a^2  }  \ll  \hat{R}\big(\hat\xi,\hat\eta\big)\hat\xi~ ,~ \hat\eta   \gg_{\hat\cg}
\end{eqnarray*}
where the term  $\ll  \hat{R}\big(\hat\xi,\hat\eta\big)\hat\xi~ ,~ \hat\eta   \gg_{\hat\cg}$ can be computed  by (\ref{eq Curvature formula e}) as it follows.\\
First we note that for $(e_{10},a)  ,   ( e_{l_0m_0},  b), (e_{l_0-m_0}  ,  b)\in\hat{\cg}$ we have
\begin{eqnarray*}
&& \ll  \hat{R}\big((e_{10},a) , ( e_{l_0m_0}   ,  b)\big)(e_{10},a)~ ,~ (   e_{l_0-m_0}  ,  b)   \gg_{\hat\cg} \\
%
%
&=&     \sum_{m,l} (-1)^m( \alpha^2+l(l+1)) \Big[ \Gamma^{lm}_{1~0a~1~0a} \Gamma^{l-m}_{l_0m_0b~ l_0-m_0b}    -  \Gamma^{lm}_{l_0m_0b~1~0a} \Gamma^{l-m}_{1~0a ~l_0-m_0b}\\
&& + G^{lm}_{1~0~l_0m_0}\Big( G^{l-m}_{1~0~l_0-m_0}k^l_{1l_0}    +  {\frac{1}{2}} \big(  b D^{l-m}_{1~0~1~0}  -a D^{l-m}_{1~0~l_0-m_0} \big)             \Big)    \Big]\\
\end{eqnarray*}
It is not difficult to see that  $\Gamma^{lm}_{1~0a~1~0a} =G^{l-m}_{1~0~1~0}=D^{l-m}_{1~0~1~0}=0$ and
\begin{equation*}
\Gamma^{lm}_{l_0m_0b~1~0a} = im_0\sqrt{\frac{3}{4\pi}}\Big(  d^l_{l_01}  + \frac{ a \sqrt{\frac{\pi}{3}}  }{ \alpha^2+l(l+1) }  \Big)\delta^l_{l_0}\delta^m_{m_0}.
\end{equation*}
Moreover we have
\begin{equation*}
\Gamma^{l-m}_{1~0a ~l_0-m_0b}  =   im_0\sqrt{\frac{3}{4\pi}}\Big( d^l_{1l_0 }  -    \frac{ a \sqrt{\frac{\pi}{3}}  }{ \alpha^2+l(l+1) }    \Big)\delta^l_{l_0}\delta^m_{m_0}
\end{equation*}
and
\begin{equation*}
G^{lm}_{1~0~l_0m_0}  =  -G^{l-m}_{1~0~l_0-m_0}   =   -im_0  \sqrt{\frac{3}{4\pi}} \delta^l_{l_0}\delta^m_{m_0}.
\end{equation*}
As a consequence $\ll  \hat{R}\big((e_{10},a) , ( e_{l_0m_0}   ,  b)\big)(e_{10},a)~ ,~ (   e_{l_0-m_0}  ,  b)   \gg_{\hat\cg} $ is
\begin{eqnarray*}
&&\sum_{m,l} (-1)^m( \alpha^2+l(l+1)) \Big[  - im_0\sqrt{\frac{3}{4\pi}}\Big( d^{l}_{l_01 }  +    \frac{ a \sqrt{\frac{\pi}{3}}  }{ \alpha^2+l(l+1) }    \Big)\\
&&\cdot im_0\sqrt{\frac{3}{4\pi}}\Big( d^l_{1l_0 }  -    \frac{ a \sqrt{\frac{\pi}{3}}  }{ \alpha^2+l(l+1) }    \Big)\delta^l_{l_0}\delta^m_{m_0}\\
&&  +   G^{lm}_{1~0~l_0m_0}\Big( G^{l-m}_{1~0~l_0-m_0}k^l_{1l_0}    -  {\frac{1}{2}}  a \sqrt{\frac{4\pi}{3}} \frac{G^{l-m}_{1~0~l_0-m_0}}{\alpha^2+l(l+1)}              \Big)    \Big] \\
\end{eqnarray*}
which is
\begin{eqnarray*}
&=&\sum_{m,l} (-1)^m( \alpha^2+l(l+1)) \Big[  m_0^2\frac{3}{4\pi}\Big( d^l_{l_01 }  +   \frac{ a \sqrt{\frac{\pi}{3}}  }{ \alpha^2+l(l+1) }    \Big)\cdot \Big( d^l_{1l_0 }  \\
&& -    \frac{ a \sqrt{\frac{\pi}{3}}  }{ \alpha^2+l(l+1) }    \Big)\delta^l_{l_0}\delta^m_{m_0} +  \underbrace{ G^{lm}_{1~0~l_0m_0} G^{l-m}_{1~0~l_0-m_0} }_{    (-im_0\sqrt{\frac{3}{4\pi}})(im_0\sqrt{\frac{3}{4\pi}})\delta^l_{l_0}\delta^m_{m_0}      }  \Big(  k^l_{1l_0}    -     \frac{ a \sqrt{\frac{\pi}{3}} }{\alpha^2+l(l+1)}              \Big)    \Big] \\
&=&      (-1)^{m_0}( \alpha^2+l_0(l_0+1))   m_0^2\frac{3}{4\pi}
\Big[  \left( d^{l_0}_{l_01 }  +    \frac{ a \sqrt{\frac{\pi}{3}}  }{ \alpha^2+l_0(l_0+1) }    \right)  \cdot \Big(  d^{l_0}_{1l_0 }    \\
&& -    \frac{ a \sqrt{\frac{\pi}{3}}  }{ \alpha^2+l_0(l_0+1) }    \Big)    +   \left(  k^{l_0}_{1l_0}    -     \frac{ a \sqrt{\frac{\pi}{3}} }{\alpha^2+l_0(l_0+1)}              \right)    \Big]\\
\end{eqnarray*}
\begin{eqnarray*}
&=&  (-1)^{m_0}(\alpha^2+ l_0(l_0+1)) m_0^2\frac{3}{4\pi}    \Big[
d^{l_0}_{l_01}d^{l_0}_{1l_0} + k^{l_0}_{1l_0}     -    \frac{ a \sqrt{\frac{\pi}{3}}  }{ \alpha^2+l_0(l_0+1) } (d^{l_0}_{l_01} -d^{l_0}_{1l_0})\\
&& -   \frac{ a^2 \frac{\pi}{3}}  { (\alpha^2+l_0(l_0+1))^2 }
-\frac{ a \sqrt{\frac{\pi}{3}}  }{ \alpha^2+l_0(l_0+1) }   \Big]\\
\end{eqnarray*}
We note that
\begin{eqnarray*}
d^{l_0}_{l_01}d^{l_0}_{1l_0} + k^l_{1l_0}  = \frac{  -(\alpha^2+2)^2 }{ 4(\alpha^2 +l_0(l_0+1))^2}  
\end{eqnarray*}
and
\begin{eqnarray*}
d^{l_0}_{l_01} -d^{l_0}_{1l_0}  = \frac{  \alpha^2+2 }{ \alpha^2 +l_0(l_0+1)}-1.  
\end{eqnarray*}
As a consequence for  the curvature term we have
\begin{eqnarray*}
&=&  (-1)^{m_0} m_0^2\frac{3}{4\pi}    \Big[  \left( \alpha^2+ l_0(l_0+1) \right) \left( d^{l_0}_{l_01}d^{l_0}_{1l_0} + k^l_{1l_0} \right)     -     a \sqrt{\frac{\pi}{3}}   \left( d^{l_0}_{l_01} -d^{l_0}_{1l_0}   \right)\\
&& -   \frac{ a^2 \frac{\pi}{3}}  { \alpha^2+l_0(l_0+1) }
- a \sqrt{\frac{\pi}{3}}    \Big]\\
&=&  (-1)^{m_0} m_0^2\frac{3}{4\pi}    \Big[    \frac{  -(\alpha^2+2)^2 }{ 4(\alpha^2 +l_0(l_0+1))}      -     a \sqrt{\frac{\pi}{3}}   \left( \frac{  \alpha^2+2 }{ \alpha^2 +l_0(l_0+1)}-1   \right)\\
&& -   \frac{ a^2 \frac{\pi}{3}}  { \alpha^2+l_0(l_0+1) }
- a \sqrt{\frac{\pi}{3}}    \Big]\\
&=&  (-1)^{m_0} m_0^2\frac{3}{4\pi}    \Big[    \frac{  -(\alpha^2+2)^2 }{ 4(\alpha^2 +l_0(l_0+1))}      -     a \sqrt{\frac{\pi}{3}}    \frac{  \alpha^2+2 }{ \alpha^2 +l_0(l_0+1)} -   \frac{ a^2 \frac{\pi}{3}}  { \alpha^2+l_0(l_0+1) }    \Big]\\
&=&  \frac{(-1)^{m_0}}{ \alpha^2+l_0(l_0+1) }   m_0^2\frac{3}{4\pi}    \Big[    -\frac{  1 }{ 4}(\alpha^2+2)^2     -     a \sqrt{\frac{\pi}{3}}   \left( \alpha^2+2 \right) -    a^2 \frac{\pi}{3}    \Big]\\
\end{eqnarray*}
In the sequel  we compute the sectional curvature $\hat\kappa(\hat\xi,\hat\eta)$ where $\hat\xi=(\nu\sqrt{\frac{4\pi}{3}}e_{1~0},\lambda)$ and $\hat{\eta}=(\eta_{l_0m_0}e_{l_0m_0} + \eta_{l_0-m_0}e_{l_0-m_0} ,0)$ with the assumption
\begin{eqnarray}\label{norm-eta=1}
\ll \hat{\eta},\hat{\eta}\gg_{\hat{\cg}}  &=& 2 \eta_{l_0m_0}\eta_{l_0-m_0}\ll e_{l_0m_0} , e_{l_0-m_0}\gg_\cg\\
&=& 2(-1)^{m_0}(\alpha^2+l_0(l_0+1))\eta_{l_0m_0}\eta_{l_0-m_0}=1 \nonumber
\end{eqnarray}
Note that
\[
\hat\xi=   (\nu\sqrt{\frac{4\pi}{3}}  e_{10},  a  )    =    \nu\sqrt{\frac{4\pi}{3}}  (e_{10},\frac{a}{\nu\sqrt{\frac{4\pi}{3}}})  :=   \nu\sqrt{\frac{4\pi}{3}} (e_{10},\lambda)
\]
with $\lambda=\frac{a}{\nu\sqrt{\frac{4\pi}{3}}}$.  Then we have
\begin{eqnarray*}
\hat{\kappa}(\hat\xi,\hat\eta)  &=&   \frac{\ll  \hat{R}\big(\hat\xi,\hat\eta\big)\hat\eta~ ,~\hat\xi    \gg_{\hat\cg}}{\ll\hat\xi,\hat\xi\gg_{\hat\cg}}\\
&=&   -  \frac{1}{    \nu^2\frac{4\pi  }{3}(\alpha^2+2)+ a^2  }  \ll  \hat{R}\big(\hat\xi,\hat\eta\big)\hat\xi~ ,~ \hat\eta   \gg_{\hat\cg}\\
&=&   -  \frac{2\nu^2\frac{4\pi}{3}\eta_{l_0m_0}\eta_{l_0-m_0}}{    \nu^2\frac{4\pi  }{3}(\alpha^2+2)+ a^2  }  \ll  \hat{R}\big((e_{10},\lambda) , ( e_{l_0m_0}   ,  0)\big)(e_{10},\lambda)~ ,~ (   e_{l_0-m_0}  ,  0)   \gg_{\hat\cg} \\
&=&-  \frac{2\nu^2\frac{4\pi}{3}\eta_{l_0m_0}\eta_{l_0-m_0}}{    \nu^2\frac{4\pi  }{3}(\alpha^2+2)+ a^2  }\cdot  \frac{(-1)^{m_0}}{ \alpha^2+l_0(l_0+1) }   m_0^2\frac{3}{4\pi}    \Big[    -\frac{  1 }{ 4}(\alpha^2+2)^2     \\
&&-     \lambda \sqrt{\frac{\pi}{3}}   \left( \alpha^2+2 \right) -    \lambda^2 \frac{\pi}{3}    \Big]
\end{eqnarray*}
\begin{eqnarray*}
&=&-  \frac{2\nu^2}{    \nu^2\frac{4\pi  }{3}(\alpha^2+2)+  a^2  }   \cdot    \frac{m_0^2}{ 2(\alpha^2+l_0(l_0+1))^2 }      \Big[    -\frac{  1 }{ 4}(\alpha^2+2)^2     \\
&&-     \lambda \sqrt{\frac{\pi}{3}}   \left( \alpha^2+2 \right) -    \lambda^2 \frac{\pi}{3}    \Big]\\
&=& -  \frac{    \nu^2m_0^2    }{    (\nu^2     \frac{4\pi  }{3}(\alpha^2+2)+ a^2)    (\alpha^2+l_0(l_0+1))^2      }             \Big[    -\frac{  1 }{ 4}(\alpha^2+2)^2     \\
&&   -     \frac{a}{\nu\sqrt{\frac{4\pi}{3}}}  \sqrt{\frac{\pi}{3}}   \left( \alpha^2+2 \right) -    \frac{a^2}{\nu^2 {\frac{4\pi}{3}}} \frac{\pi}{3}    \Big]\\
&=& -  \frac{    m_0^2    }{   ( \nu^2     \frac{4\pi  }{3}(\alpha^2+2)+  a^2)    (\alpha^2+l_0(l_0+1))^2      }             \Big[    -\frac{  \nu^2 }{ 4}(\alpha^2+2)^2     -      \frac{\nu a}{2}   \left( \alpha^2+2 \right)   -     \frac{a^2}{4}    \Big]\\
\end{eqnarray*}
\begin{eqnarray*}
&=&   \frac{    m_0^2    }{    (\nu^2     \frac{4\pi  }{3}(\alpha^2+2)+  a^2)    (\alpha^2+l_0(l_0+1))^2      }             \Big[    \frac{  \nu^2 }{ 4}(\alpha^2+2)^2     +      \frac{\nu a}{2}   \left( \alpha^2+2 \right)   +     \frac{ a^2}{4}    \Big]\\
\end{eqnarray*}

\textbf{Acknowledgements}. Funded by the Deutsche Forschungsgemeinschaft (DFG, German
Research Foundation) - 517512794.

%
%
%

%
Ali Suri, Universit\"{a}t Paderborn, Warburger Str.\ 100,
33098 Paderborn, Germany; asuri@math.upb.de\vfill
\end{document}